\documentclass[12pt]{amsart}
\usepackage{amscd,amsmath,amssymb,amsfonts}
\usepackage[cmtip, all]{xy}

\newtheorem{thm}{Theorem}[section]
\newtheorem{prop}[thm]{Proposition}
\newtheorem{lem}[thm]{Lemma}
\newtheorem{cor}[thm]{Corollary}

\theoremstyle{definition}

\newtheorem{defn}[thm]{Definition}
\newtheorem{remk}[thm]{Remark}
\newtheorem{remks}[thm]{Remarks}

\newtheorem{exm}[thm]{Example}
\newtheorem{exms}[thm]{Examples}
\newtheorem{notat}[thm]{Notation}
\newtheorem{claim}[thm]{Claim}
\numberwithin{equation}{section}

{\hfill$\square$\end{defn}}

{\hfill$\square$\end{remk}}

{\hfill$\square$\end{remks}}

{\hfill$\square$\end{exm}}

{\hfill$\square$\end{exms}}

{\hfill$\square$\end{notat}}

\newcommand{\sC}{{\mathcal C}}

\newcommand{\sH}{{\mathcal H}}

\newcommand{\sM}{{\mathcal M}}

\newcommand{\sO}{{\mathcal O}}

\newcommand{\sR}{{\mathcal R}}
\newcommand{\sS}{{\mathcal S}}

\newcommand{\sV}{{\mathcal V}}

\newcommand{\sZ}{{\mathcal Z}}
\newcommand{\A}{{\mathbb A}}

\newcommand{\C}{{\mathbb C}}

\newcommand{\G}{{\mathbb G}}

\newcommand{\bL}{{\mathbb L}}

\newcommand{\N}{{\mathbb N}}
\renewcommand{\P}{{\mathbb P}}
\newcommand{\Q}{{\mathbb Q}}

\newcommand{\Z}{{\mathbb Z}}

\newcommand{\surj}{\twoheadrightarrow}
\newcommand{\inj}{\hookrightarrow}

\newcommand{\Hom}{{\rm Hom}}

\newcommand{\0}{\emptyset}

\newcommand{\ds}{{/\kern-3pt/}}

\newcommand{\un}{\underline}
\newcommand{\ov}{\overline}
\newcommand{\wt}{\widetilde}

\renewcommand{\dim}{\text{\rm dim}}

\newcommand{\tuborg}{\left\{\begin{array}{ll}}
\newcommand{\sluttuborg}{\end{array}\right.}

\begin{document}
\title{Equivariant cobordism of schemes}
\author{Amalendu Krishna}
\address{School of Mathematics, Tata Institute of Fundamental Research,  
Homi Bhabha Road, Colaba, Mumbai, India}
\email{amal@math.tifr.res.in}

\baselineskip=10pt 
  

\begin{abstract}
Let $k$ be a field of characteristic zero. For a linear algebraic group $G$ 
over $k$ acting on a $k$-scheme $X$, we define the equivariant algebraic 
cobordism of $X$, which extends the definition of equivariant cobordism of 
smooth schemes in \cite{DD} to all schemes. Using a refined version of the 
Levine-Morel localization, we establish the localization sequence for the
equivariant cobordism for all $G$-schemes. We explicitly describe the relation
of equivariant cobordism with equivariant Chow groups, $K$-groups and
complex cobordism. 

We show that the rational equivariant cobordism of a
$G$-variety can be expressed as the Weyl group invariants of the equivariant
cobordism for the action of a maximal torus of $G$.  
As applications, we show that the rational algebraic 
cobordism of the classifying space of a complex algebraic group
is isomorphic to its complex cobordism. 
\end{abstract}
\maketitle

\section{Introduction}
Let $k$ be a field of characteristic zero. 
Based on the construction of the motivic algebraic cobordism spectrum
$MGL$ by Voevodsky, Levine and Morel \cite{LM} gave a geometric 
construction of the algebraic cobordism and showed that this is a universal 
oriented Borel-Moore homology theory in the category of varieties over the 
field $k$. Their definition was extended by Deshpande \cite{DD} in the
equivariant set-up that led to the notion of the equivariant cobordism
of smooth varieties acted upon by linear algebraic groups. This in particular
allowed one to define the algebraic cobordism of the classifying spaces
analogous to their complex cobordism. 

Apart from its many applications 
in the equivariant set up which are parallel to the ones in the
non-equivariant world, an equivariant cohomology theory often leads to
the description of the corresponding non-equivariant cohomology by mixing the 
geometry of the variety with the representation theory of the underlying groups.

Our aim in this first part of a series of papers is to develop the  
theory of equivariant cobordism in the category of all $k$-schemes
with action of a linear algebraic group. We establish the fundamental
properties of this theory and give applications. In the second part 
\cite{Krishna5} of this series, we shall give other important applications of 
the results of this paper. Some further applications of the results of this
paper to the computation of the non-equivariant cobordism rings appear in
\cite{Krishna3} and \cite{KU}. 
We now describe some of the main results in of this paper. 

Let $G$ be a linear algebraic group over $k$. In this paper, a scheme will
mean a quasi-projective $k$-scheme and all $G$-actions will be assumed to be 
linear. If $X$ is a smooth scheme with a $G$-action, Deshpande defined the
equivariant cobordism $\Omega^G_*(X)$ using the coniveau filtration on
the Levine-Morel cobordism of certain smooth mixed spaces.
This was based on the construction of the Chow groups of classifying spaces
in \cite{Totaro1} and the equivariant Chow groups in \cite{EG}.
 
Using a {\sl niveau} filtration on the algebraic cobordism, which is based
on the analogous filtration on any Borel-Moore homology theory as 
described in \cite[Section~3]{BO}, we define the equivariant algebraic
cobordism of any $k$-scheme with $G$-action in Section~\ref{section:EC}.
This is defined by taking a projective limit over the quotients of the 
Levine-Morel cobordism of certain mixed spaces by various levels of the niveau 
filtration. In order to make sense of this construction, one needs to prove
various properties of the above niveau filtration which is done in
Section~\ref{section:NIV}. These equivariant cobordism groups coincide with
the one in \cite{DD} for smooth schemes. 
We also show in Section~5 how one can recover the formula for the cobordism 
group of certain classifying spaces directly from the above definition, by 
choosing suitable models for the underlying mixed spaces.

The first main result on these equivariant cobordism groups is to show 
that they satisfy the localization exact sequence, a fundamental property
of any cohomology theory. For smooth schemes, a version
of the localization sequence was proven in \cite[Theorem~7.3]{DD}. However,
the proof given there is not complete ({\sl cf.} Remark~\ref{remk:incomplete}).
The problem stems from the fact that the projective limit is not in 
general a right exact functor. Hence to prove the equivariant version of
the localization sequence, one needs a stronger version of the Levine-Morel
localization sequence involving the above niveau filtration.
We prove this refined localization sequence in Theorem~\ref{thm:RLS}  which is
then used to give the complete proof of the equivariant localization sequence in
Theorem~\ref{thm:ELS}. We prove the other expected properties
such as functoriality, homotopy invariance, exterior product,
projection formula and existence of Chern classes for equivariant vector 
bundles in Theorem~\ref{thm:Basic}.

In Section~7, we show how the equivariant cobordism is related to other
equivariant cohomology theories such as equivariant Chow groups,
equivariant $K$-groups and equivariant complex cobordism.
Using some properties of the niveau filtration and
known relation between the non-equivariant cobordism and Chow groups,
we deduce an explicit formula ({\sl cf.} Proposition~\ref{prop:CBCH})
which relates the equivariant cobordism and the equivariant Chow groups of 
$k$-varieties. Using this and the main results of \cite{Krishna}, we
give a formula in Theorem~\ref{thm:CKtheory} which relates the equivariant 
cobordism with the equivariant $K$-theory of smooth schemes.
We also construct a natural transformation from the algebraic to the
equivariant version the complex cobordism for varieties over the field
of complex numbers. 

Our next main result of this paper is Theorem~\ref{thm:W-inv}, where
we show that for a connected linear algebraic group $G$ acting on a scheme $X$,
there is a canonical isomorphism $\Omega^G_*(X) \xrightarrow{\cong}
\left(\Omega^T_*(X)\right)^W$ with rational coefficients, where $T$ is a split 
maximal torus of a Levi subgroup of $G$ with Weyl group $W$. 
This is mainly achieved by the Morita isomorphism of
Proposition~\ref{prop:Morita} and a detour to the motivic cobordism
$MGL$ and its extension $MGL'$ to singular schemes by Levine \cite{Levine1}.
This motivic cobordism $MGL'$ comes into play due to the advantage that
it is a bigraded Borel-Moore homology which has long exact localization
sequences. What helps us in using this theory in our context is the recent
comparison result of Levine \cite{Levine2} which shows that the Levine-Morel
cobordism theory is a piece of the more general $MGL'$-theory.

As an easy consequence of Proposition~\ref{prop:CBCH}, we recover 
Totaro's cycle class map ({\sl cf.} \cite{Totaro1}) 
\[
CH^*(BG) \to MU^*(BG) \otimes_{\bL} \Z \to H^*(BG)
\]
for a complex linear algebraic group $G$. 
It is conjectured that this map is an isomorphism
of rings. This conjecture has been shown to be true by Totaro for some
classical groups such as $BGL_n$, $O_n$, $Sp_{2n}$ and $SO_{2n+1}$.
Although, we can not say anything about this conjecture here, we do show
as a consequence of  Theorem~\ref{thm:W-inv} that the map 
$CH^*(BG) \to MU^*(BG) \otimes_{\bL} \Z$ is indeed an isomorphism of rings
with the rational coefficients (see Theorem~\ref{thm:ACC*} for the full
statement). We do this by first showing that there is a natural ring
homomorphism $\Omega^*(BG) \to MU^*(BG)$ (with integer coefficients)
which lifts Totaro's map. We then show that this map is in fact an 
isomorphism with rational coefficients using Theorem~\ref{thm:W-inv}.

As some more applications of Theorem~\ref{thm:W-inv}, we aim to compute the 
cobordism ring of certain spherical varieties in \cite{KK}.  

\section{Recollection of algebraic cobordism}\label{section:AC}
In this section, we briefly recall the definition of algebraic cobordism
of Levine-Morel. We also recall the other definition of this object as given
by Levine-Pandharipande. Since we shall be concerned with the study of 
schemes with group actions and the associated quotient schemes, and since 
such quotients often require the original scheme to be quasi-projective,
we shall assume throughout this paper that all schemes over $k$ are 
quasi-projective. 
\\

\noindent
{\bf Notations.} We shall denote the category of quasi-projective $k$-schemes 
by $\sV_k$. By a scheme, we shall mean an object 
of $\sV_k$. The category of smooth quasi-projective schemes
will be denoted by $\sV^S_k$. If $G$ is a linear algebraic group over $k$, we 
shall denote the category of quasi-projective $k$-schemes with a $G$-action
and $G$-equivariant maps by $\sV_G$. The associated category of smooth
$G$-schemes will be denoted by $\sV_G^S$. All $G$-actions in this paper will be
assumed to be linear. Recall that this means that all $G$-schemes are
assumed to admit $G$-equivariant ample line bundles. This assumption is
always satisfied for normal schemes ({\sl cf.} \cite[Theorem~2.5]{Sumihiro}, 
\cite[5.7]{Thomason1}).

\subsection{Algebraic cobordism}
Before we define the algebraic cobordism, we recall the Lazard ring $\bL$.
It is a polynomial ring over $\Z$ on infinite but countably many variables and 
is given by the quotient of the polynomial ring $\Z[A_{ij}| (i,j) \in \N^2]$ 
by the relations, which uniquely define the universal formal group law 
$F_{\bL}$ of rank one on $\bL$. This formal group law is given by the power 
series
\[
F_{\bL}(u,v) = u + v + {\underset {i,j \ge 1} \sum} a_{ij} u^iv^j,
\]
where $a_{ij}$ is the equivalence class of $A_{ij}$ in the ring $\bL$.
The Lazard ring is graded by setting the degree of $a_{ij}$ to be $1-i-j$.
In particular, one has $\bL_{0} = \Z, \bL_{-1} = \Z a_{11}$ and $\bL_{i} = 0$
for $i \ge 1$, that is, $\bL$ is non-positively graded. We shall write $\bL_*$
for the graded ring such that $\bL_{*,i} = \bL_{-i}$ for $i \in \Z$.
We now define the algebraic cobordism of Levine and Morel \cite{LM}.
 
Let $X$ be an equi-dimensional $k$-scheme. A cobordism cycle over $X$ is
a family $\alpha = [Y \xrightarrow{f} X, L_1, \cdots , L_r]$, where $Y$ is a
smooth scheme, the map $f$ is projective, and $L_i$'s are line bundles on $Y$.
Here, one allows the set of line bundles to be empty.
The degree of such a cobordism cycle is  defined to be ${\rm deg}(\alpha) =
{\rm dim}_k(Y)-r$ and its codimension is defined to be ${\rm dim}(X) -
{\rm deg}(\alpha)$. Let $\sZ^*(X)$ be the free abelian group generated by the 
cobordism cycles of the above type. Note that this group is graded by the
codimension of the cycles. In particular, for $j \in \Z$, $\sZ^j(X)$
is the free abelian group on cobordism cycles $\alpha =
[Y \xrightarrow{f} X, L_1, \cdots , L_r]$, where $Y$ is smooth and irreducible
and codimension of $\alpha$ is $j$.   

We impose several relations on $\sZ^*(X)$ in order to define the algebraic
cobordism group. The first among these is the so called 
{\sl dimension axiom}: let $\sR^*_{\rm dim}(X)$ be the graded 
subgroup of $\sZ^*(X)$ generated by the cobordism cycles $\alpha = 
[Y \xrightarrow{f} X, L_1, \cdots , L_r]$ such that ${\rm dim}_k Y < r$.
Let 
\[
{\sZ}^*_{\rm dim}(X) = \frac{\sZ^*(X)}{\sR^*_{\rm dim}(X)}.
\]
For a line bundle $L$ on $X$ and cobordism cycle $\alpha$ as above, we define 
the Chern class operator on
${\sZ}^*_{\rm dim}(X)$ by letting $c_1(L)(\alpha) = [Y \xrightarrow{f} X,
L_1, \cdots , L_r, f^*(L)]$. 
 
Next, we impose the so called {\sl section axiom}. Let $\sR^*_{\rm sec}(X)$ be 
the graded subgroup of ${\sZ}^*_{\rm dim}(X)$ generated
by cobordism cycles of the form $[Y \to X, L] - [Z \to X]$, where
$Y \xrightarrow{s} L$ is a section of the line bundle $L$ on $Y$ which is
transverse to the zero-section, and $Z \inj Y$ is the closed subvariety of
$Y$ defined by the zeros of $s$. The transversality of $s$ ensures that 
$Z$ is a smooth variety. In particular, $[Z \to X]$ is a well-defined
cobordism cycle on $X$. Define 
\[
\un{\Omega}^*(X) = \frac{{\sZ}^*_{\rm dim}(X)}{\sR^*_{\rm sec}(X)}.
\]
The assignment $X \mapsto \un{\Omega}^*(X)$ is called the pre-cobordism
theory.

Finally, we impose the {\sl formal group law} on the cobordism using the 
following relation. For $X$ as above, let
$\sR^*_{\rm FGL}(X) \subset \bL {\otimes}_{\Z} \un{\Omega}^*(X)$ be the graded 
$\bL$-submodule generated by elements of the form 
\[
\left\{F_{\bL}\left(c_1(L), c_1(M)\right) (x) - c_1(L \otimes M) (x)|
x \in \un{\Omega}^*(X), L, M \in {\rm Pic}(X)\right\}.
\]
We define the {\sl algebraic cobordism} group of $X$ by
\begin{equation}\label{eqn:CB}
\Omega^*(X) = \frac{\bL {\otimes}_{\Z} \un{\Omega}^*(X)}{\sR^*_{\rm FGL}(X)}.
\end{equation}

If $X$ is not necessarily equi-dimensional, we define $\sZ_*(X)$ to be same as
$\sZ^*(X)$ except that $\sZ_*(X)$ is now graded by the degree of the cobordism
cycles. In particular, $\sZ_i(X)$ is the free abelian group on cobordism
cycles $[Y \xrightarrow{f} X, L_1, \cdots , L_r]$ such that $f$ is projective
and $Y$ is smooth and irreducible such that ${\rm dim}(Y) - r = i$.
One then defines $\Omega_*(X)$ to be the quotient of
$\bL_* \otimes_{\Z} \un{\Omega}_*(X)$ in the same way as above. 
Note that for $X$ equi-dimensional of dimension $d$ and $i \in \Z$, one has 
$\Omega^i(X) \cong \Omega_{d-i}(X)$.

Observe that $\Omega^*(X)$ is a graded $\bL$-module such that $\Omega^j(X)
= 0$ for $j > {\rm dim}(X)$ and $\Omega^j(X)$ can be non-zero for any given
$-\infty < j \le {\dim}(X)$. Similarly, $\Omega_*(X)$ is a graded 
$\bL_*$-module which has no component in the negative degrees and it can be
non-zero in arbitrarily large positive degree.
  
The following is the main result of Levine and Morel from which most of their
other results on algebraic cobordism are deduced. We refer to 
{\sl loc. cit.} for more properties.
\begin{thm}\label{thm:Levine-M}
The functor $X \mapsto \Omega_*(X)$ is the universal Borel-Moore homology
on the category $\sV_k$. In other words, it is universal among the homology
theories on $\sV_k$ which have functorial push-forward for projective 
morphism, pull-back for smooth morphism (any morphism of smooth schemes), 
Chern classes for line bundles, and which satisfy Projective bundle formula, 
homotopy invariance, the above dimension, section and formal group law axioms.
Moreover, for a $k$-scheme $X$ and closed subscheme $Z$ of $X$  with open 
complement $U$, there is a localization exact sequence
\[
\Omega_{*}(Z) \to \Omega_*(X) \to \Omega_*(U) \to 0.
\]
\end{thm}
It was also shown in {\sl loc. cit.} that the natural composite map
\[
\Phi : \bL \to \bL \otimes_{\Z} \un{\Omega}^*(k) \surj \Omega^*(k)
\]
\[
a \mapsto [a]
\]
is an isomorphism of commutative graded rings.

As an immediate corollary of Theorem~\ref{thm:Levine-M}, we see that for
a smooth $k$-scheme $X$ and an embedding $\sigma : k \to \C$, there is a 
natural morphism of graded rings
\begin{equation}\label{eqn:A-C-C}
\Phi^{top}_X : \Omega^*(X) \to MU^{2*}(X_{\sigma}(\C)),
\end{equation} 
where $MU^*(X_{\sigma}(\C))$ is the complex cobordism ring of the complex
manifold $X_{\sigma}(\C)$ given by the complex points of $X \times_k \C$.
This map is an isomorphism for $X = {\rm Spec}(k)$.
In particular, there are isomorphisms of graded rings
\begin{equation}\label{eqn:A-C-C*}
\bL \xrightarrow{\cong} \Omega^*(k) \xrightarrow{\cong} 
MU^{2*} \xrightarrow{\cong} MU^*,
\end{equation}
where $MU^*$ is the complex cobordism ring of a point.
As a corollary, we see that for any field extension $k \inj K$, the 
natural map $\Omega^*(k) \to \Omega^*(K)$ is an isomorphism.

\subsection{Cobordism via double point degeneration}
To enforce the formal group law on the algebraic
cobordism in order to make it an oriented cohomology theory on the
category of smooth varieties, Levine and Morel artificially imposed this
condition by tensoring their pre-cobordism theory with the Lazard ring.
Although they were able to show that the resulting map
$\sZ_*(X) \to \Omega_*(X)$ is still surjective, they were unable to
describe the explicit geometric relations in $\sZ_*(X)$ that define 
$\Omega_*(X)$. This was subsequently accomplished by Levine-Pandharipande
\cite{LP}. We conclude our introduction to the algebraic cobordism by briefly
discussing the construction of Levine-Pandharipande.
For $n \ge 1$, let $\square^n$ denote the space ${(\P^1_k - \{1\})}^n$.

\begin{defn}\label{defn:DPD}
A morphism $Y \xrightarrow{\pi} \square^1$ is called a {\sl double point   
degeneration}, if $Y$ is a smooth scheme and ${\pi}^{-1}(0)$ is 
scheme-theoretically given as the union $A \cup B$, where $A$ and $B$
are smooth divisors on $Y$ which intersect transversely. 
The intersection $D = A \cap B$ is called the double point locus of $\pi$.
Here, $A$, $B$ and $D$ are allowed to be disconnected or, even empty.
\end{defn}

For a double point degeneration as above, notice that the scheme $D$ is 
also smooth and $\sO_D(A+B)$ is trivial. In particular, one sees that
$N_{A/D} \otimes_D N_{B/D} \cong \sO_D$. This is turn implies that the 
projective bundles $\P(\sO_D \oplus N_{A/D}) \to D$ and
$\P(\sO_D \oplus N_{B/D}) \to D$ are isomorphic, where 
$N_{A/D}$ and $N_{B/D}$ are the normal bundles of $D$ in $A$ and
$B$ respectively. Let $\P(\pi) \to D$ denote any of these two projective
bundles.

Let $X$ be a $k$-scheme and let $Y \xrightarrow{f} X \times \square^1$ be
a projective morphism from a smooth scheme $Y$. Assume that the composite
map $\pi : Y \to X \times \square^1 \to \square^1$ is a double point 
degeneration such that $Y_{\infty} = {\pi}^{-1}(\infty)$ is smooth.
We define the cobordism cycle on $X$ associated to the morphism $f$ to be the 
cycle
\begin{equation}\label{eqn:C-cycle}
C(f) = [Y_{\infty} \to X] - [A \to X] - [B \to X] + [\P(\pi) \to X].
\end{equation}
Let $\sM_*(X)$ be the free abelian group on the isomorphism classes of the
morphisms $[Y \xrightarrow{f} X]$, where $Y$ is smooth and irreducible and
$f$ is projective. Then ${\sM}_*(X)$ is a graded abelian group, where the 
grading is by the dimension of $Y$. Let ${\sR}_*(X)$ be the subgroup
of $\sM_*(X)$ generated by all cobordism cycles $C(f)$, where $C(f)$ is as in
~\eqref{eqn:C-cycle}. Note that ${\sR}_*(X)$ is a graded subgroup of 
${\sM}_*(X)$. Define
\begin{equation}\label{eqn:C-cycle*} 
\omega_*(X) = \frac{{\sM}_*(X)}{{\sR}_*(X)}.
\end{equation}
\begin{thm}[\cite{LP}]\label{thm:Levine-P}
There is a canonical isomorphism 
\begin{equation}\label{eqn:C-cycle*1} 
\omega_*(X) \xrightarrow{\cong} \Omega_*(X)
\end{equation}
of oriented Borel-Moore homology theories on $\sV$.
\end{thm}

\section{Niveau filtration on algebraic cobordism}\label{section:NIV}
In this section, we introduce the niveau filtration on the algebraic cobordism 
which plays an important role in the definition of the equivariant
algebraic cobordism. Our main result here is a refined localization sequence
for the cobordism which preserves the niveau filtration. This new localization
sequence will have interesting consequences in the study of the
equivariant cobordism.

Let $X$ be a $k$-scheme of dimension $d$. For $j \in \Z$, let $Z_j$ be the
set of all closed subschemes $Z \subset X$ such that ${\rm dim}_k(Z) \le j$
(we assume ${\rm dim}(\0) = - \infty$). The set $Z_j$ is then ordered by 
the inclusion. For $i \ge 0$,  we define
\[
\Omega_i(Z_j) = {\underset{Z \in Z_j} \varinjlim} \Omega_i(Z) \ \ {\rm and} 
\ \ {\rm put} 
\]
\[
\Omega_*(Z_j) = {\underset{i \ge 0} \bigoplus} \ \Omega_i(Z_j).
\]
It is immediate that $\Omega_*(Z_j)$ is a graded $\bL_*$-module and there is
a graded $\bL_*$-linear map $\Omega_*(Z_j) \to \Omega_*(X)$.

Following \cite[Section~3]{BO}, we let ${Z_j}/{Z_{j-1}}$ denote the ordered set
of pairs $(Z, Z') \in Z_j \times Z_{j-1}$ such that $Z' \subset Z$ with the
ordering 
\[
(Z, Z') \ge (Z_1, Z_1') \ \ {\rm if} \ \ Z_1 \subseteq Z \ \ {\rm and} \ \
Z_1' \subseteq Z'.
\]
We let 
\[
\Omega_*\left({Z_j}/{Z_{j-1}}(X)\right) := 
{\underset{(Z, Z') \in {Z_j}/{Z_{j-1}}} \varinjlim} \Omega_i(Z-Z').
\]
\begin{lem}\label{lem:Niv*1}
For $f : X' \to X$ projective, the push-forward map $\Omega_*(X') 
\xrightarrow{f_*} \Omega_*(X)$ induces a push-forward map
$\Omega_*\left({Z_j}/{Z_{j-1}}(X')\right) \to 
\Omega_*\left({Z_j}/{Z_{j-1}}(X)\right)$.
\end{lem}
\begin{proof}
Let $(Z, Z') \in {Z_j}/{Z_{j-1}}(X')$. Then $(W, W') = ({\rm Im}(Z),
{\rm Im}(Z')) \in {Z_j}/{Z_{j-1}}(X)$. It suffices now to show that
$f_*$ induces a natural map $\Omega_*(Z-Z') \to \Omega_*(W-W')$.
However, this follows directly from the localization exact sequences
\[
\xymatrix@C.5pc{
\Omega_*(Z') \ar[r] \ar[d]_{f_*} & \Omega_*(Z) \ar[r] \ar[d]^{f_*} &
\Omega_*(Z- Z') \ar[r] \ar@{.>}[d] & 0 \\
\Omega_*(W') \ar[r] & \Omega_*(W) \ar[r] &
\Omega_*(W- W') \ar[r] & 0}
\]
and the fact that the square on the left is commutative.
\end{proof}
For $x \in Z_j$, let 
\begin{equation}\label{eqn:Niv*2} 
\widetilde{\Omega_*(k(x))} = {\underset{U \subseteq \ov{\{x\}}} \varinjlim} 
\Omega_*(U),
\end{equation}
where the limit is taken over all non-empty open subsets of $\ov{\{x\}}$.
Taking the limit over the localization sequences
\[
\Omega_*(Z') \to \Omega_*(Z) \to \Omega_*(Z-Z') \to 0
\]
for $(Z, Z') \in {Z_j}/{Z_{j-1}}$, one now gets an exact sequence 
\begin{equation}\label{eqn:Niv*3} 
\Omega_*(Z_{j-1}) \to \Omega_*(Z_j) \to 
{\underset{x \in  (Z_j - Z_{j-1})} \bigoplus} \widetilde{\Omega_*(k(x))} \to 0.
\end{equation}

\begin{defn}\label{defn:niveau}
We define $F_j\Omega_*(X)$ to be the image of the natural $\bL_*$-linear map
$\Omega_*(Z_j) \to \Omega_*(X)$. In other words, $F_j\Omega_*(X)$ is the image
of all $\Omega_*(W) \to \Omega_*(X)$, where $W \to X$ is a projective map
such that ${\rm dim}({\rm Image}(W)) \le j$.
\end{defn}
One checks at once that there is a canonical {\sl niveau filtration}
\begin{equation}\label{eqn:niveau1}
0 = F_{-1}\Omega_*(X) \subseteq F_0\Omega_*(X) \subseteq \cdots \subseteq
F_{d-1}\Omega_*(X) \subseteq F_d\Omega_*(X) = \Omega_*(X).
\end{equation}

\begin{lem}\label{lem:Niv*}
If $f : X' \to X$ is a projective morphism, then 
$f_*\left(F_j\Omega_*(X')\right) \subseteq F_j\Omega_*(X)$. If $g: X' \to X$
is a smooth morphism of relative dimension $r$, then 
$g^*\left(F_j\Omega_*(X)\right) \subseteq F_{j+r}\Omega_*(X')$.
\end{lem}
\begin{proof}
The first assertion is obvious from the definition.
In fact, the push-forward map preserves the niveau filtration at the level
of the free abelian groups of cobordism cycles. The second assertion
also follows immediately using the fact that for a cobordism cycle 
$[Y \to X]$, one has $g^*\left([Y \to X]\right) = [Y \times_X X' \to X']$.
This in turn implies that $g^* \circ f_* = f'_* \circ g'^*$ for a Cartesian
square
\[
\xymatrix@C.7pc{
 W' \ar[r]^{f'} \ar[d]_{g'} & X' \ar[d]^{g} \\
W \ar[r]_{f} & X}
\]
such that $f$ is projective and $g$ is smooth.
\end{proof}
\begin{thm}{(Refined localization sequence)}\label{thm:RLS}
Let $X$ be a $k$-scheme and let $Z$ be a closed subscheme of $X$ with the
complement $U$. Then for every $j \in \Z$, there is an exact sequence
\[
F_j\Omega_*(Z) \to F_j\Omega_*(X) \to F_j\Omega_*(U) \to 0.
\]
\end{thm}
\begin{proof}
All these groups are zero if $j < 0$, so we assume $j \ge 0$.
Let $F_j\sZ_*(X)$ be the free abelian group on cobordism cycles 
$[Y \xrightarrow{f} X]$ such that $Y$ is irreducible and $\dim (f(Y)) \le j$. 
Note that $f(Y)$
is a closed and irreducible subscheme of $X$ since $Y$ is irreducible and
$f$ is projective. It is then clear that $F_j\sZ_*(X) \subset
\sZ_*(X)$ and $F_j\sZ_*(X) \surj F_j\Omega_*(X)$. We first show the 
following.
\begin{claim}\label{claim:RLS1}
\[
{\rm Ker}\left(F_j\sZ_*(X) \to F_j\Omega_*(X)\right) \to
{\rm Ker}\left(F_j\sZ_*(U) \to F_j\Omega_*(U)\right) 
\]
is surjective.
\end{claim}
{\sl Proof of the claim :} In one of the steps in the proof of the localization
sequence in \cite{LM}, it is shown that the map
\begin{equation}\label{eqn:RLS2}
{\rm Ker}\left(\sZ_*(X) \to \Omega_*(X)\right) \to
{\rm Ker}\left(\sZ_*(U) \to \Omega_*(U)\right) 
\end{equation}
is surjective. 
Let $\alpha = \stackrel{n}{\underset{l=1}\sum}a_l[W_l \xrightarrow{f_l} U]$ be 
an element in the kernel of the map $F_j\sZ_*(U) \to F_j\Omega_*(U)$.
By ~\eqref{eqn:RLS2}, there is an element ${\beta}
= \stackrel{m}{\underset{s=1}\sum}b_{s}[Y_s \xrightarrow{g_s} X] \in
\sZ_*(X)$ which restricts to $\alpha$. We can assume that any two
summands of this sum are not isomorphic. For a map $Y \to X$, let $Y_U \to U$
be its restriction to $U$. Then we get 
\[
\stackrel{n}{\underset{l=1}\sum}a_l[W_l \xrightarrow{f_l} U]
= \stackrel{m}{\underset{s=1}\sum}b_{s}[{(Y_s)}_U \xrightarrow{g_s} U].
\]
Since $\sZ_*(U)$ is a free abelian group, we see that $m = n$ and for
each $1 \le s \le n$, $[{(Y_s)}_U \xrightarrow{g_s} U] \cong
[W_l \xrightarrow{f_l} U]$ for some $l$. In particular, 
$\dim\left(g_s|_U({(Y_s)}_U)\right) \le j$. Since $Y_s$ is irreducible and $g_s$
is projective, one easily checks that $\dim(g_s(Y_s)) \le j$. 
That is $\beta \in F_j\sZ_*(X)$. This proves the claim. 

Since $F_j\sZ_*(X)$ is a free abelian group,
we see using Theorem~\ref{thm:Levine-P}, Claim~\ref{claim:RLS1} and the 
surjection $F_j\sZ_*(X) \surj F_j\Omega_*(X)$ that 
${\rm Ker}\left(F_j\Omega_*(X) \to F_j\Omega_*(U)\right)$ is generated by the 
cobordism cycles $\alpha = 
[Y \xrightarrow{f} X] - [Y' \xrightarrow{f'} X]$, where
$Y$ and $Y'$ are smooth and irreducible and $Y_U \cong Y'_U$ such that 
the dimension of their images are at most $j$.

Put $S = f(Y), \ S' = f'(Y'), \ T = S \cap U$ and $T' = S' \cap U$.
Then the irreducibility of the sources and the projectivity of the maps
imply that all these are irreducible closed subschemes of dimension 
at most $j$. Put $W = S \cup S'$ and $V = T \cup T'$. Then it is
easy to see that $[Y \to X]$ and $[Y' \to X]$ are the images of the
cycles $[Y \to W]$ and $[Y' \to W]$ in $\sZ_*(W)$ under the push-forward via the
closed immersion $W \inj X$. Moreover, $Y_V \cong Y'_V$. 
Hence we see that we have
a cycle $\beta = [Y \to W] - [Y' \to W] \in \sZ_*(W)$ such that
$[Y_V \to V] \cong [Y'_V \to V]$. Now, it follows from the proof of the
localization sequence in {\sl loc. cit.} that $\beta \in
{\rm Image}\left(\sZ_*(W \cap Z) \to \sZ_*(W)\right)$.
In particular, $\alpha = {\rm Image}(\beta) \in {\rm Image}
\left(F_j\sZ_*(Z) \to F_j\sZ_*(X)\right)$. Hence we have shown that
${\rm Ker}\left(F_j\Omega_*(X) \to F_j\Omega_*(U)\right)$ comes from
$F_j\Omega_*(Z)$. We now show that the map $F_j\Omega_*(X) \to
F_j\Omega_*(U)$ is surjective to complete the proof of the theorem.

We can assume that $X$ is irreducible and $j < \dim(X)$. By the generalized
degree formula ({\sl cf.} \cite[Theorem~4.4.7]{LM}), any 
$\alpha \in F_j\Omega_*(U)$ can be written as
\[
\alpha = \stackrel{s}{\underset{l = 1} \sum} u_l [\wt{U}_l \to U],
\]
where $\wt{U}_l \to U_l$ is a resolution of singularities of an irreducible
closed subscheme $U_l \subset U$ and $u_l \in \bL_*$ 
Then we must have $[\wt{U}_l \to U] \in
F_j\Omega_*(U)$. Letting $X_l = \ov{U_l} \subset X$, we see that
$\ov {\alpha} = \stackrel{s}{\underset{l = 1} \sum} u_l [\wt{X}_l \to X]
 \in F_j\Omega_*(X)$ and its image in $\Omega_*(U)$ is $\alpha$.
This finishes the proof of the theorem.
\end{proof}

The following is an immediate consequence of Theorem~\ref{thm:RLS}.
\begin{cor}\label{cor:Niv-open}
Let $X$ be a $k$-scheme. Then for any $j \ge 0$ and any closed subscheme
$Z \subset X$ of dimension at most $j$, the natural map
$\Omega_*(X) \to \Omega_*(X - Z)$ induces an isomorphism
\[
\frac{\Omega_*(X)}{F_j\Omega_*(X)} \xrightarrow{\cong}
\frac{\Omega_*(X-Z)}{F_j\Omega_*(X-Z)}.
\]
\end{cor}

\begin{lem}\label{lem:Niv-Chow}
For a $k$-scheme $X$ and $i \ge 0$, the natural map $\Omega_i(X) \to CH_i(X)$ 
has the factorization
\[
\Omega_i(X) \to \frac{\Omega_i(X)}{F_{i-1}\Omega_i(X)} \to CH_i(X).
\]
\end{lem}
\begin{proof}
By Theorem~\ref{thm:Levine-P}, $\Omega_*(X)$ is generated by the cobordism
cycles $[Y \to X]$, where $Y$ is smooth and $f$ is projective.
It follows from the definition of the niveau filtration that
$F_{j}\Omega_*(X)$ is generated by the cobordism cycles of the form
$i_*\left([Y \to Z]\right)$, where $Z \overset{\phi}{\inj} X$ is a closed 
subscheme of $X$ of dimension at most $j$. Since $\Omega_* \to CH_*$
is a natural transformation of oriented Borel-Moore homology theories,
we get a commutative diagram
\[
\xymatrix@C.8pc{
\Omega_i(Z) \ar[r] \ar[d]_{\phi_*} & CH_i(Z) \ar[d]^{\phi '_*} \\
\Omega_i(X) \ar[r] & CH_i(X).}
\]
The lemma now follows from the fact that $CH_i(Z) = 0$ if $j \le i-1$.
\end{proof}

\begin{lem}\label{lem:Niv-Hi}
Let $E \xrightarrow{f} X$ be a vector bundle of rank $r$. Then the 
pull-back map $f^*: \Omega_*(X) \to \Omega_*(E)$ induces an isomorphism
\[
F_j\Omega_*(X) \xrightarrow{\cong} F_{j+r}\Omega_*(E)
\]
for all $j \in \Z$. In particular, $F_{<r}\Omega_*(E) = 0$.
\end{lem}
\begin{remk} The reader should be warned that the map $f^*$ shifts the degree 
of the grading by $r$.
\end{remk}
\begin{proof}
Using the generalized degree formula, this can be proved in the same way 
as \cite[Lemma~3.3]{DD}, where a similar result is proven for smooth
varieties and coniveau filtration. We sketch the proof in the singular case.
 
By the homotopy invariance of the algebraic cobordism, the natural map
$\Omega_*(X) \xrightarrow{f^*} \Omega_*(E)$ is an isomorphism.
So we only need to show that this map is surjective at each level of the
niveau filtration. So let $e \in F_j\Omega_*(E)$. Assume that the dimension of
$X$ is $d$. Since $F_{d+r}\Omega_*(E) = \Omega_*(E)$, the homotopy invariance
implies that we can assume that $j < d+r$. The generalized degree formula
now implies that there are irreducible closed subschemes 
$\{E_1, \cdots , E_s\}$ of
$E$ such that in the $\bL_*$-module $\Omega_*(E)$, one has
\[
e = \stackrel{s}{\underset{l = 1} \sum} w_l [\widetilde{E_l} \to E],
\]
where $\wt{E_l} \to E_l$ is a resolution of singularities of $E_l$ and
$w_l \in \bL_*$. In particular, we see that for each $l$,
$[\widetilde{E_l} \to E] \in \Omega_{p_l}(E)$ for some  $p_l \le j$. 

On the other hand, the isomorphism of $f^*$ and the generalized degree formula
for $\Omega_*(X)$  imply that for each $l$, 
\[
[\wt{E_l} \to E] = 
\stackrel{n_l}{\underset{n = 1} \sum} 
w_{l,n} \left[\left(\wt{X}_{l,n} \times_{X_{l,n}} E \right)\to E \right],
\]
where $\wt{X}_{l,n} \to X_{l,n}$ is a resolution of singularities of a closed
subscheme $X_{l,n}$ of $X$. Since each component of this sum is a
homogeneous element, we must have that 
$\left[ \left(\wt{X}_{l,n} \times_{X_{l,n}} E\right) \to E \right]
\in \Omega_{p_l}(E)$, that is, $[\wt{X}_{l,n} \to X] \in \Omega_{p_l-r}(X)$.
This in turn implies that $x_l =
\stackrel{n_l}{\underset{n = 1} \sum} w_{l,n} [\wt{X}_{l,n} \to X] \in
\Omega_{p_l-r}(X)$. Hence we get 
\[
x = 
\stackrel{s}{\underset{l = 1} \sum} w_l x_l \in F_{j-r}\Omega_*(X) \ \ 
{\rm and} \ \ e = f^*(x).
\]
\end{proof}

\section{Equivariant algebraic cobordism}\label{section:EC}
In this text, $G$ will denote a linear algebraic group of dimension $g$ 
over $k$. All representations of $G$ will be finite dimensional. 
The definition of equivariant cobordism needs one to consider certain kind of 
mixed spaces which in general may not be a scheme even if the original space 
is a scheme. The following well known ({\sl cf.} \cite[Proposition~23]{EG}) 
lemma shows that this
problem does not occur in our context and all the mixed spaces in this
paper are schemes with ample line bundles.
\begin{lem}\label{lem:sch}
Let $H$ be a linear algebraic group acting freely and linearly on a 
$k$-scheme $U$ such that the quotient $U/H$ exists as a quasi-projective
variety. Let $X$ be a $k$-scheme with a linear action of $H$.
Then the mixed quotient $X \stackrel{H} {\times} U$ exists for the 
diagonal action of $H$ on $X \times U$ and is quasi-projective.
Moreover, this quotient is smooth if both $U$ and $X$ are so.
In particular, if $H$ is a closed subgroup of a linear algebraic group $G$ 
and $X$ is a $k$-scheme with a linear action of $H$, then the quotient 
$G \stackrel{H} {\times} X$ is a quasi-projective scheme.
\end{lem}
\begin{proof} It is already shown in \cite[Proposition~23]{EG} using
\cite[Proposition~7.1]{GIT} that the quotient $X \stackrel{H} {\times} U$ 
is a scheme. Moreover, as $U/H$ is quasi-projective, 
\cite[Proposition~7.1]{GIT} in fact shows that $X \stackrel{H} {\times} U$ 
is also quasi-projective. The similar conclusion about 
$G \stackrel{H} {\times} X$ follows from the first case by taking $U = G$
and by observing that $G/H$ is a smooth quasi-projective scheme
({\sl cf.} \cite[Theorem~6.8]{Borel}). The assertion about the smoothness
is clear since $X \times U \to X \stackrel{H} {\times} U$ is a principal
$H$-bundle.
\end{proof}   
For any integer $j \ge 0$, let $V_j$ be an $l$-dimensional representation of 
$G$ and let $U_j$ 
be a $G$-invariant open subset of $V_j$ such that the codimension of the 
complement $(V_j-U_j)$ in $V_j$ is at least $j$ and $G$ acts freely on 
$U_j$ such that the quotient ${U_j}/G$ is a quasi-projective scheme. Such a 
pair $\left(V_j,U_j\right)$ will be called a {\sl good} pair for the $G$-action
corresponding to $j$ ({\sl cf.} \cite[Section~2]{Krishna}).
It is easy to see that a good pair always exists ({\sl cf.}
\cite[Lemma~9]{EG}). Let $X_G$ denote the mixed quotient 
$X \stackrel{G} {\times} U_j$ of the product $X \times U_j$ by the 
diagonal action of $G$, which is free.

Let $X$ be a $k$-scheme of dimension $d$ with a $G$-action.
Fix $j \ge 0$ and let $(V_j, U_j)$ be an $l$-dimensional good pair 
corresponding to $j$. For $i \in \Z$, set
\begin{equation}\label{eqn:E-cob*}
{\Omega^G_i(X)}_j =  \frac{\Omega_{i+l-g}\left({X\stackrel{G} {\times} U_j}\right)}
{F_{d+l-g-j}\Omega_{i+l-g}\left({X\stackrel{G} {\times} U_j}\right)}.
\end{equation}
\begin{lem}\label{lem:ECob1}
For a fixed $j \ge 0$, the group ${\Omega^G_i(X)}_j$ is independent of the 
choice of the good pair $(V_j, U_j)$.
\end{lem}
\begin{proof}
Let $(V'_j, U'_j)$ be another good pair of dimension $l'$ corresponding to 
$j$. Following the proof of a similar result for the equivariant
Chow groups in \cite[Proposition~1]{EG}, we let $V = V_j \oplus V'_j$ and
$U = (U_j \oplus V'_j) \cup (V_j \oplus U'_j)$. Let $G$ act diagonally on
$V$. Then it is easy to see that the complement of the open subset
$X \stackrel{G} {\times} (U_j \oplus V'_j)$ in $X \stackrel{G} {\times} U$ has 
dimension at most $d+l+l'-g-j$. Hence by Corollary~\ref{cor:Niv-open}, the map
\[
\frac{\Omega_{i+l+l'-g}\left(X \stackrel{G} {\times} U\right)}
{F_{d+l+l'-g-j}\Omega_{i+l+l'-g}\left(X \stackrel{G} {\times} U\right)}
\to
\frac{\Omega_{i+l+l'-g}\left(X \stackrel{G} {\times} (U_j \oplus V'_j)\right)}
{F_{d+l+l'-g-j}\Omega_{i+l+l'-g}\left(X \stackrel{G} 
{\times} (U_j \oplus V'_j)\right)}
\]
is an isomorphism.
On the other hand, the map $X \stackrel{G} {\times} (U_j \oplus V'_j)
\to X \stackrel{G} {\times} U_j$ is a vector bundle of rank $l'$ and hence
by Lemma~\ref{lem:Niv-Hi}, the map
\[
\frac{\Omega_{i+l-g}\left(X \stackrel{G} {\times} U_j \right)}
{F_{d+l-g-j}\Omega_{i+l-g}\left(X \stackrel{G} 
{\times} U_j \right)}  \to
\frac{\Omega_{i+l+l'-g}\left(X \stackrel{G} {\times} (U_j \oplus V'_j)\right)}
{F_{d+l+l'-g-j}\Omega_{i+l+l'-g}\left(X \stackrel{G} 
{\times} (U_j \oplus V'_j)\right)}
\]
is also an isomorphism. Combining the above two isomorphisms, we get the
isomorphism
\[
\frac{\Omega_{i+l+l'-g}\left(X \stackrel{G} {\times} U\right)}
{F_{d+l+l'-g-j}\Omega_{i+l+l'-g}\left(X \stackrel{G} {\times} U\right)}
\cong \frac{\Omega_{i+l-g}\left(X \stackrel{G} {\times} U_j \right)}
{F_{d+l-g-j}\Omega_{i+l-g}\left(X \stackrel{G} 
{\times} U_j \right)}.
\]
In the same way, we also get an isomorphism
\[
\frac{\Omega_{i+l+l'-g}\left(X \stackrel{G} {\times} U\right)}
{F_{d+l+l'-g-j}\Omega_{i+l+l'-g}\left(X \stackrel{G} {\times} U\right)}
\cong 
\frac{\Omega_{i+l'-g}\left(X \stackrel{G} {\times} U'_j \right)}
{F_{d+l'-g-j}\Omega_{i+l'-g}\left(X \stackrel{G} 
{\times} U'_j \right)},
\]
which proves the result.
\end{proof}
\begin{lem}\label{lem:ECob2}
For $j' \ge j \ge 0$, there is a natural surjective map
$\Omega^G_i(X)_{j'} \surj \Omega^G_i(X)_j$.
\end{lem}
\begin{proof} 
Choose a good pair $(V_{j'}, U_{j'})$ for $j'$. Then it is clearly a good pair
for $j$ too. Moreover, there is a natural surjection
\[
 \frac{\Omega_{i+l-g}\left({X\stackrel{G} {\times} U_{j'}}\right)}
{F_{d+l-g-j'}\Omega_{i+l-g}\left({X\stackrel{G} {\times} U_{j'}}\right)}
\surj
 \frac{\Omega_{i+l-g}\left({X\stackrel{G} {\times} U_{j'}}\right)}
{F_{d+l-g-j}\Omega_{i+l-g}\left({X\stackrel{G} {\times} U_{j'}}\right)}.
\]
On the other hand, the left and the right terms are $\Omega^G_i(X)_{j'}$
and $\Omega^G_i(X)_{j}$ respectively by Lemma~\ref{lem:ECob1}.
\end{proof}

\begin{defn}\label{defn:ECob}
Let $X$ be a $k$-scheme of dimension $d$ with a $G$-action. For any 
$i \in \Z$, we define the {\sl equivariant algebraic cobordism} of $X$ to be 
\[
\Omega^G_i(X) = {\underset {j} \varprojlim} \ \Omega^G_i(X)_j.
\]
\end{defn}
The reader should note from the above definition that unlike the ordinary
cobordism, the equivariant algebraic cobordism $\Omega^G_i(X)$ can be 
non-zero for any $i \in \Z$. We set 
\[
\Omega^G_*(X) = {\underset{i \in \Z} \bigoplus} \ \Omega^G_i(X).
\]
If $X$ is an equi-dimensional $k$-scheme with $G$-action, we let
$\Omega^i_G(X) = \Omega^G_{d-i}(X)$ and $\Omega^*_G(X) =
{\underset{i \in \Z} \oplus} \ \Omega^i_G(X)$. We shall denote the 
equivariant cobordism $\Omega^*_G(k)$ of the ground field by $\Omega^*(BG)$.
It is also called the algebraic cobordism of the classifying space of $G$.

\begin{remk}\label{remk:G-trivial}
If $G$ is the trivial group, we can take the good pair $(V_j, V_j)$ for every
$j$ where $V_j$ is any $l$-dimensional $k$-vector space. In that case,
we get $\Omega_{i+l}(X \stackrel{G}{\times} V_j) \cong 
\Omega_{i+l}(X {\times} V_j)$ which is isomorphic to $\Omega_{i}(X)$ by the
homotopy invariance of the non-equivariant cobordism. Moreover,
$F_{d+l-j}\Omega_{i+l}(X \times V_j)$ is isomorphic to $F_{d-j}\Omega_i(X)$
by Lemma~\ref{lem:Niv-Hi} and this last term is zero for all large $j$.
In particular, we see from ~\eqref{eqn:E-cob*} and the definition of the
equivariant cobordism that there is a canonical isomorphism 
$\Omega^G_*(X) \xrightarrow{\cong} \Omega_*(X)$.
\end{remk}
\begin{remk}\label{remk:Csmooth}
It is easy to check from the above definition of the niveau filtration 
that if $X$ is a smooth and irreducible $k$-scheme of dimension $d$,
then $F_j\Omega_i(X) = F^{d-j}\Omega^{d-i}(X)$, where $F^{\bullet}\Omega^*(X)$ 
is the coniveau filtration used in \cite{DD}. Furthermore, one also 
checks in this case that if $G$ acts on $X$, then
\begin{equation}\label{eqn:Csmooth1}
\Omega^i_G(X) = {\underset {j} \varprojlim} \
\frac{\Omega^{i}\left({X\stackrel{G} {\times} U_j}\right)}
{F^{j}\Omega^{i}\left({X\stackrel{G} {\times} U_j}\right)},
\end{equation}
where $(V_j, U_j)$ is a good pair corresponding to any $j \ge 0$.
Thus the above definition~\ref{defn:ECob} of the equivariant cobordism
coincides with that of \cite{DD} for smooth schemes.
\end{remk}
\begin{remk}\label{remk:C-Chow}
As is evident from the above definition (see Example~\ref{exms:BT}), the 
equivariant cobordism 
$\Omega^G_i(X)$ can not in general be computed in terms of the algebraic
cobordism of one single mixed space. This makes these groups more
complicated to compute than the equivariant Chow groups, which can be
computed in terms of a single mixed space. This also motivates one to ask
if the equivariant cobordism can be defined in such a way that they can
be calculated using one single mixed space in a given degree.
It follows however from Lemmas~\ref{lem:ECob1} and ~\ref{lem:ECob2} that for a
given $i$ and $j$, each component of the projective system
${\left\{\Omega^G_i(X)_j\right\}}_{j \ge 0}$ can be computed using a single
mixed space.
\end{remk} 
\subsection{Change of groups}
If $H \subset G$ is a closed subgroup of dimension $h$, then any 
$l$-dimensional good pair 
$(V_j, U_j)$ for $G$-action is also a good pair for the induced $H$-action. 
Moreover, for any $X \in \sV_G$ of dimension $d$, 
$X \stackrel{H}{\times} U_j \to X \stackrel{G}{\times} U_j$
is an \'etale locally trivial $G/H$-fibration and hence a smooth map 
({\sl cf.} \cite[Theorem~6.8]{Borel}) of relative dimension $g-h$. 
This induces the inverse system of pull-back maps
 \[
\Omega^G_i(X)_j = \frac{\Omega_{i+l-g}\left({X\stackrel{G} {\times} U_{j}}\right)}
{F_{d+l-g-j}\Omega_{i+l-g}\left({X\stackrel{G} {\times} U_{j}}\right)}
\to \frac{\Omega_{i+l-h}\left({X\stackrel{H} {\times} U_{j}}\right)}
{F_{d+l-h-j}\Omega_{i+l-h}\left({X\stackrel{H} {\times} U_{j}}\right)}
= \Omega^H_i(X)_j 
\]
and hence a natural restriction map
\begin{equation}\label{eqn:res}
r^G_{H,X} : \Omega^G_*(X) \to \Omega^H_*(X).
\end{equation}
Taking $H = \{1\}$ and using Remark~\ref{remk:G-trivial}, we get the
{\sl forgetful} map 
\begin{equation}\label{eqn:resf}
r^G_X : \Omega^G_*(X) \to \Omega_*(X)
\end{equation}
from the equivariant to the non-equivariant cobordism. 
Since $r^G_{H,X}$ is obtained as a pull-back under the smooth map, it commutes
with any projective push-forward and smooth pull-back ({\sl cf.} 
Theorem~\ref{thm:Basic}). We remark here that
although the definition of $r^G_{H,X}$ uses a good pair, it is easy to see
as in Lemma~\ref{lem:ECob1} that it is independent of the choice of such
good pairs. 

\subsection{Fundamental class of cobordism cycles}
Let $X \in \sV_G$ and let $Y \xrightarrow{f} X$ be a morphism in $\sV_G$
such that $Y$ is smooth of dimension $d$ and $f$ is projective. For any
$j \ge 0$ and any $l$-dimensional good pair $(V_j, U_j)$, 
$[Y_G \xrightarrow{f_G} X_G]$ is an ordinary cobordism cycle of dimension
$d+l-g$ by Lemma~\ref{lem:Proj-Q} and hence defines an element 
$\alpha_j \in {\Omega^G_d(X)}_j$. Moreover, it is evident that the image of 
$\alpha_{j'}$ is $\alpha_j$ for $j' \ge j$. Hence we get a unique element
$\alpha \in \Omega^G_d(X)$, called the {\sl G-equivariant fundamental class}
of the cobordism cycle $[Y \xrightarrow{f} X]$. We also see from this
more generally that if $[Y \xrightarrow{f} X, L_1, \cdots, L_r]$ is as
above with each $L_i$ a $G$-equivariant line bundle on $Y$, then this defines
a unique class in $\Omega^G_{d-r}(X)$. It is interesting question to ask under 
what conditions on the group $G$, the equivariant cobordism group 
$\Omega^G_*(X)$ is generated by the fundamental classes of $G$-equivariant
cobordism cycles on $X$. It turns out that this question indeed has a positive
answer if $G$ is a split torus by \cite[Theorem~4.11]{Krishna5}.

\section{Localization sequence and other properties}
\label{section:BasicP}
In this section, we establish the localization sequence for the equivariant
cobordism using Theorem~\ref{thm:RLS}. 
We also prove other basic properties of the equivariant algebraic
cobordism which are analogous to the non-equivariant case. 

\begin{thm}\label{thm:ELS}
Let $X$ be a $G$-scheme of dimension $d$ and let $Z \subset X$ be a 
$G$-invariant closed subscheme with the complement $U$. 
Then there is an exact sequence
\[
\Omega^G_*(Z) \to \Omega^G_*(X) \to \Omega^G_*(U) \to 0.
\]
\end{thm}
\begin{proof}
We fix integers $i \in \Z$ and $j \ge 0$ and first show that the sequence
\begin{equation}\label{eqn:ELS1}
{\Omega^G_i(Z)}_j \to {\Omega^G_i(X)}_j \to {\Omega^G_i(U)}_j \to 0
\end{equation}
is exact.
Choose a good pair $(V_j, U_j)$ of dimension $l$ for $j$. Then we see that 
$Z_G \subset X_G$ is a closed subscheme with the complement $U_G$. 
Hence by applying Theorem~\ref{thm:RLS} at the appropriate levels of the 
niveau filtration and taking the quotients, we get an exact sequence
\[
\frac{\Omega_{i+l-g}(Z_G)}{F_{d+l-g-j} \Omega_{i+l-g}(Z_G)}
\to \frac{\Omega_{i+l-g}(X_G)}{F_{d+l-g-j} \Omega_{i+l-g}(X_G)} \to
\frac{\Omega_{i+l-g}(U_G)}{F_{d+l-g-j} \Omega_{i+l-g}(U_G)} \to 0.
\]
If $d' = \dim(Z)$, then $F_{d'+l-g-j} \Omega_{i+l-g}(Z_G) \subset 
F_{d+l-g-j} \Omega_{i+l-g}(Z_G)$ and hence by Lemma~\ref{lem:ECob1},
we get an exact sequence
\begin{equation}\label{eqn:ELS2}
{\Omega^G_i(Z)}_j \to {\Omega^G_i(X)}_j \to 
{\Omega^G_i(U)}_j \to 0.
\end{equation}
Now, we let for any $j' \ge 0$, $\ov{F_{j'}\Omega_{i+l-g}(Z_G)}
= {\rm Image}\left(F_{j'}\Omega_{i+l-g}(Z_G) \to F_{j'}\Omega_{i+l-g}(X_G)\right)$ and
$E_{j'} =  
{\rm Ker}\left(F_{j'}\Omega_{i+l-g}(Z_G) \to F_{j'}\Omega_{i+l-g}(X_G)\right)$.
This gives a filtration 
\[
0 = E_{-1} \subseteq E_0 \subseteq \cdots \subseteq E_{d'} = E 
= {\rm Ker}\left(\Omega_{i+l-g}(Z_G) \to \Omega_{i+l-g}(X_G)\right).
\]
Then by Theorem~\ref{thm:RLS}, we get for every $j' \ge 0$, the exact
sequences
\begin{equation}\label{eqn:ELS3}
0 \to \frac{E}{E_{j'}} \to \frac{\Omega_{i+l-g}(Z_G)}{F_{j'}\Omega_{i+l-g}(Z_G)}
\to \frac{\ov{\Omega_{i+l-g}(Z_G)}}{\ov{F_{j'}\Omega_{i+l-g}(Z_G)}} \to 0,
\end{equation}
\begin{equation}\label{eqn:ELS4}
0 \to \frac{\ov{\Omega_{i+l-g}(Z_G)}}{\ov{F_{j'}\Omega_{i+l-g}(Z_G)}} \to
\frac{\Omega_{i+l-g}(X_G)}{F_{j'}\Omega_{i+l-g}(X_G)} \to 
\frac{\Omega_{i+l-g}(U_G)}{F_{j'}\Omega_{i+l-g}(U_G)} \to 0.
\end{equation}
Since $\{E/{E_{j'}}\}$ and 
$\left\{\frac{\ov{\Omega_{i+l-g}(Z_G)}}{\ov{F_{j'}\Omega_{i+l-g}(Z_G)}}\right\}$
are the inverse systems of surjective maps and hence satisfy the 
Mittag-Leffler(ML) condition, we identify the various terms in the exact 
sequences by Lemma~\ref{lem:ECob1} to get an exact sequence
\begin{equation}\label{eqn:ELS5}
{\underset{j} \varprojlim} \
\frac{\Omega_{i+l-g}(Z_G)}{F_{d+l-g-j}\Omega_{i+l-g}(Z_G)} \to 
{\underset{j} \varprojlim} \ {\Omega^G_i(X)}_j \to
{\underset{j} \varprojlim} \ {\Omega^G_i(U)}_j \to 0.
\end{equation}
Finally, the exact sequences
\[
0 \to \frac{F_{d+l-g-j}\Omega_{i+l-g}(Z_G)}{F_{d'+l-g-j}\Omega_{i+l-g}(Z_G)} 
\to \frac{\Omega_{i+l-g}(Z_G)}{F_{d'+l-g-j}\Omega_{i+l-g}(Z_G)} \to 
\frac{\Omega_{i+l-g}(Z_G)}{F_{d+l-g-j}\Omega_{i+l-g}(Z_G)} \to 0
\]
and the Mittag-Leffler property of the first terms imply that
\[
{\underset{j} \varprojlim} \ {\Omega^G_i(Z)}_j \surj 
{\underset{j} \varprojlim} \ 
\frac{\Omega_{i+l-g}(Z_G)}{F_{d+l-g-j}\Omega_{i+l-g}(Z_G)}.
\]
Combining this with ~\eqref{eqn:ELS5}, we get the desired
exact sequence
\[
{\underset{j} \varprojlim} \ {\Omega^G_i(Z)}_j \to  
{\underset{j} \varprojlim} \ {\Omega^G_i(X)}_j \to 
{\underset{j} \varprojlim} \ {\Omega^G_i(U)}_j \to 0,  
\]
which proves the theorem.
\end{proof}
\begin{remk}\label{remk:incomplete}
If $X$ is a smooth scheme, a proof of the above localization
sequence is given in \cite[Theorem~7.3]{DD}. However, the proof given there
is not complete, since it only proves the exact sequence ~\eqref{eqn:ELS2} for
a fixed $j$ by a different method. As the reader will observe, this
exact sequence does not imply the localization theorem. The crucial ingredient
in the above proof (including the direct proof of ~\eqref{eqn:ELS2}) is the
refined localization sequence of Theorem~\ref{thm:RLS}.
\end{remk}

Before we prove the other properties of the equivariant cobordism,
we mention the following elementary result.
\begin{lem}\label{lem:Proj-Q}
Let $f : X \to Y$ be a projective $G$-equivariant map in $\sV_G$ such that 
$G$ acts freely on $X$ and $Y$. Then the induced map $\ov{f} :
X/G \to Y/G$ of quotients is also projective.
\end{lem}
\begin{proof}
It follows from our assumption and Lemma~\ref{lem:sch} that $\ov{f}:
X' = X/G \to Y/G = Y'$ is a morphism in $\sV_k$. Furthermore, the square
\[
\xymatrix@C.9pc{
X \ar[r] \ar[d]_{f} & X' \ar[d]^{\ov{f}} \\
Y \ar[r] & Y'}
\]
is Cartesian.
Since both the horizontal maps are the principal $G$-bundles, they are 
smooth and surjective. Since proper maps have smooth descent (in fact
fpqc descent), we see that $\ov{f} : X' \to Y'$ is proper. Since
these schemes are quasi-projective, we leave it as an exercise to show that
$\ov{f}$ is also quasi-projective and hence must be projective.
\end{proof}

\begin{thm}\label{thm:Basic}
The equivariant algebraic cobordism satisfies the following properties. \\
$(i)$ {\sl Functoriality :} The assignment $X \mapsto \Omega_*(X)$ is
covariant for projective maps and contravariant for smooth maps in
$\sV_G$. It is also contravariant for l.c.i. morphisms in $\sV_G$. 
Moreover, for a fiber diagram 
\[
\xymatrix@C.7pc{
X' \ar[r]^{g'} \ar[d]_{f'} & X \ar[d]^{f} \\
Y' \ar[r]_{g} & Y}
\]
in $\sV_G$ with $f$ projective and $g$ smooth, one has 
$g^* \circ f_* = {f'}_* \circ {g'}^* : \Omega^G_*(X) \to \Omega^G_*(Y')$.
\\
$(ii) \ Homotopy :$  If $f : E \to X$ is a $G$-equivariant vector bundle,
then $f^*: \Omega^G_*(X) \xrightarrow{\cong} \Omega^G_*(E)$. \\
$(iii) \ Chern \ classes :$ For any $G$-equivariant vector bundle $E
\xrightarrow{f} X$ of rank $r$, there are equivariant Chern class operators
$c^G_m(E) : \Omega^G_*(X) \to \Omega^G_{*-m}(X)$ for $0 \le m \le r$ with
$c^G_0(E) = 1$. These Chern classes 
have same functoriality properties as in the non-equivariant case. 
Moreover, they satisfy the Whitney sum formula. \\
$(iv) \ Free \ action :$ If $G$ acts freely on $X$ with quotient $Y$, then
$\Omega^G_*(X) \xrightarrow{\cong} \Omega_*(Y)$. \\
$(v) \ Exterior \ Product :$ There is a natural product map
\[
\Omega^G_i(X) \otimes_{\Z} \Omega^G_{i'}(X') \to \Omega^G_{i+i'}(X \times X').
\]
In particular, $\Omega^G_*(k)$ is a graded algebra and $\Omega^G_*(X)$ is
a graded $\Omega^G_*(k)$-module for every $X \in \sV_G$. \\
$(vi) \ Projection \ formula :$ For a projective map $f : X' \to X$ in
$\sV^S_G$, one has for $x \in \Omega^G_*(X)$ and $x' \in \Omega^G_*(X')$,
the formula : $f_*\left(x' \cdot f^*(x)\right) = f_*(x') \cdot x$. \\
\end{thm}
\begin{proof}
Assume that the dimensions of $X$ and $Y$ are $m$ and $n$ respectively
and let $d = m-n$ be the relative dimension of a projective $G$-equivariant
morphism $f: X \to Y$.
For a fixed $j \ge 0$, let $(V_j, U_j)$ be an $l$-dimensional good pair for
$j$. 
Since $f$ is projective, Lemma~\ref{lem:Proj-Q} implies that
$\ov{f} : X_G \to Y_G$ is projective and hence by Theorem~\ref{thm:Levine-M} and
Lemma~\ref{lem:Niv*}, there is a push-forward map
\[
\frac{\Omega_{i+l-g}(X_G)}{F_{m+l-g-j}\Omega_{i+l-g}(X_G)}
\to \frac{\Omega_{i+l-g}(Y_G)}{F_{m+l-g-j}\Omega_{i+l-g}(Y_G)}
= \frac{\Omega_{i+l-g}(Y_G)}{F_{n+l-g-(j-d)}\Omega_{i+l-g}(Y_G)}.
\]
In particular, we get a compatible system of maps
\[
{\Omega^G_i(X)}_{j+d} \to {\Omega^G_i(Y)}_j.
\]
Taking the inverse limits, one gets the desired push-forward map
$\Omega^G_i(X) \xrightarrow{f_*} \Omega^G_i(Y)$.

If $f$ is smooth of relative dimension $d$, then $\ov{f} : X_G \to Y_G$
is also smooth of same relative dimension. Hence, we get a compatible system of
pull-back maps ${\Omega^G_i(Y)}_j \xrightarrow{\ov{f}}^* {\Omega^G_{i+d}(X)}_j$.
Taking the inverse limit, we get the desired pull-back map of the
equivariant algebraic cobordism groups. If $f$ is a l.c.i. morphism of 
$G$-schemes, the same proof applies using the existence of similar map in the
non-equivariant case. The required commutativity of the pull-back and 
push-forward maps follows exactly in the same way from the corresponding result
for the non-equivariant cobordism groups. 

To prove the homotopy property, let $E \xrightarrow{f} X$ be a $G$-equivariant
vector bundle of rank $r$. For any $j \ge 0$, let $(V_j, U_j)$ be a 
good pair for $j$. Then the map of mixed quotients $E_G \to X_G$ is
a vector bundle of rank $r$ ({\sl cf.} \cite[Lemma~1]{EG}). Hence by 
Lemma~\ref{lem:Niv-Hi}, the pull-back map ${\Omega^G_i(X)}_j \to
{\Omega^G_{i+r}(E)}_j$ is an isomorphism. If $j' \ge j$, then we can choose
a common good pair for both $j$ and $j'$. Hence, we a pull-back map
of the inverse systems $\{{\Omega^G_i(X)}_j\} \to \{{\Omega^G_{i+r}(E)}_j\}$
which is an isomorphism at each level. Hence $f^* : \Omega^G_i(X) \to
\Omega^G_{i+r}(E)$ is an isomorphism.

To define the Chern classes of an equivariant vector bundle $E$ of rank
$r$, we choose an $l$-dimensional good pair $(V_j,U_j)$ and
consider the vector bundle $E_G \to X_G$ as above and let
$c^G_{m, j} : \Omega_{i+l-g}(X_G) \to \Omega_{i+l-g-m}(X_G)$ be the non-equivariant
Chern class as in \cite[4.1.7]{LM}.  
For a closed subscheme $Z {\overset{\iota}\inj} X_G$, the projection
formula for the non-equivariant cobordism 
\begin{equation}\label{eqn:PF-bundle}
c^G_{m,j}(E_G)\circ {\iota}_*  = {\iota}_* \circ 
\left(c^G_{m,j}\left({\iota}^*(E_G)\right)\right)
\end{equation}
implies that $c^G_{m,j}(E_G)$ descends to maps
$c^G_{m,j} : {\Omega^G_{i}(X)}_{j} \to {\Omega^G_{i-m}(X)}_j$.

One shows as in Lemma~\ref{lem:ECob1}
that this is independent of the choice of the good pairs. Furthermore,
choosing a common good pair for $j' \ge j$, we see that 
$c^G_{m, j}$ actually defines a map of the inverse systems. Taking the 
inverse limit, we get the Chern classes $c^G_m(E) :
\Omega^G_i(X) \to \Omega^G_{i-m}(X)$ for $0 \le m \le r$ with $c^G_0(E) = 1$.
The functoriality and the Whitney sum formula for the equivariant
Chern classes are easily proved along the above lines using the
analogous properties of the non-equivariant Chern classes. 

The statement about the free action follows from \cite[Lemma~7.2]{DD} and
Remark~\ref{remk:G-trivial}. 

We now show the existence of the exterior product of the equivariant cobordism
which requires some work. Let $d$ and $d'$ be the dimensions of
$X$ and $X'$ respectively. We first define maps
\begin{equation}\label{eqn:EP1}
{\Omega^G_i(X)}_j \otimes {\Omega^G_{i'}(X')}_{j} \to
 {\Omega^G_{i+i'}(X \times X')}_{j} \ \ {\rm for} \ j \ge 0.
\end{equation}
Let $(V_j, U_j)$ be an $l$-dimensional good pair for $j$ and let $\alpha = 
[Y \xrightarrow{f} X]$ and ${\alpha}' = [Y' \xrightarrow{f'} X']$ be the 
cobordism cycles on $X_G$ and $X'_G$ respectively. Using the fact that 
$X \times U_j \to X_G$ and $X'\times U_j \to X'_G$ are principal $G$-bundles,
we get the unique cobordism cycles $[\wt{Y} \to X \times U_j]$ and
$[\wt{Y'} \to X' \times U_j]$ whose $G$-quotients are the above chosen cycles.
We define $\alpha \star {\alpha}' = [\wt{Y} \stackrel{G}{\times} \wt{Y'} 
\to {(X \times X')}_G]$. Note that $(V_j \times V_j, U_j \times U_j)$ is 
a good pair for $j$ of dimension $2l$ and ${(X \times X')}_G$ is the quotient 
of $X \times X' \times U_j \times U_j$ for the free diagonal action of $G$ and
$\alpha \star {\alpha}'$ is a well defined cobordism cycle by 
Lemma~\ref{lem:Proj-Q}.

Suppose now that $W \xrightarrow{p} X_G \times \square^1$ is
a projective morphism from a smooth scheme $W$ such that the composite
map $\pi : W \to X_G \times \square^1 \to \square^1$ is a double point 
degeneration with $W_{\infty} = {\pi}^{-1}(\infty)$ smooth. Letting $G$ act
trivially on $\square^1$, this gives a unique $G$-equivariant double
point degeneration $\wt{W} \xrightarrow{\wt{p}} X \times U_j \times 
\square^1$ of $G$-schemes.
This implies in particular that $\wt{W} \times \wt{Y'} 
\xrightarrow{\wt{p} \times \wt{f'}} X \times X' \times
U_j \times U_j \times \square^1$ is also a $G$-equivariant double point
degeneration whose quotient for the free $G$-action gives a double
point degeneration $\wt{W} \stackrel{G}{\times} \wt{Y'} \xrightarrow{q}
{(X \times X')}_G \times \square^1$. Moreover, it is easy to see from this
that $C(p) \star {\alpha}' = C(q)$ ({\sl cf.} ~\eqref{eqn:C-cycle}). 
Reversing the roles of $X$ and $X'$ and
using ~\eqref{eqn:C-cycle} and Theorem~\ref{thm:Levine-P}, we get the maps
\[
\Omega_{i+l-g}(X_G) \otimes \Omega_{i'+l-g}(X'_G) \to
\Omega_{i+i'-2l-g}\left({(X \times X')}_G\right).
\]
It is also clear from the definition of $\alpha \star {\alpha}'$ and the 
niveau filtration that 
\[
\left\{F_{d+l-g-j}\Omega_{i+l-g}(X_G) \otimes  \Omega_{i'+l-g}(X'_G)\right\}
+ \left\{\Omega_{i+l-g}(X_G) \otimes F_{d'+l-g-j}\Omega_{i'+l-g}(X'_G)\right\}
\]
\[
\hspace*{8cm} \to
F_{d+d'+2l-g}\Omega_{i+i'-2l-g}\left({(X \times X')}_G\right).
\]
This defines the maps as in ~\eqref{eqn:EP1}. One can now show as in 
Lemma~\ref{lem:ECob1} that these maps are independent of the choice
of the good pairs. We get the desired exterior product as the composite
map
\begin{equation}\label{eqn:EP2}
\Omega^G_i(X) \otimes_{\Z} \Omega^G_{i'}(X') = 
{\underset{j} \varprojlim} \ {\Omega^G_i(X)}_j \otimes_{\Z}
{\underset{j} \varprojlim}  \ {\Omega^G_{i'}(X')}_{j} \to
{\underset{j} \varprojlim} \ \left({\Omega^G_i(X)}_j \otimes_{\Z} 
{\Omega^G_{i'}(X')}_{j}\right)
\end{equation}
\[
\hspace*{4cm} \to {\underset{j} \varprojlim} \ {\Omega^G_{i+i'}(X \times X')}_{j}
= \Omega^G_{i+i'}(X \times X').
\]
Finally for $X$ smooth, we get the product structure on $\Omega^*_G(X)$
via the composite $\Omega^*_G(X) \otimes_{\Z} \Omega^*_G(X) \to
\Omega^*_G(X \times X) \xrightarrow{\Delta^*_X} \Omega^*_G(X)$.  The
projection formula can now be proven by using the non-equivariant version of 
such a formula ({\sl cf.} \cite[5.1.4]{LM}) at each level of the projective 
system $\{{\Omega^i_G(X)}_j\}$ and then taking the inverse limit.
\end{proof}
\begin{prop}[Morita Isomorphism]\label{prop:Morita}
Let $H \subset G$ be a closed subgroup and let $X \in {\sV}_H$.
Then there is a canonical isomorphism
\begin{equation}\label{eqn:MoritaI}
\Omega^G_*\left(G \stackrel{H}{\times} X\right) \xrightarrow{\cong}
\Omega^H_*(X).
\end{equation}
\end{prop}
\begin{proof} Define an action of $H \times G$ on $G \times X$ by
\[
(h, g) \cdot (g', x) = \left(gg'h^{-1}, hx\right),
\]
and an action of $H \times G$ on $X$ by $(h, g) \cdot x = hx$. 
Then the projection map $G \times X \xrightarrow{p} X$ is 
$\left(H \times G\right)$-equivariant which is a 
$G$-torsor. Hence by \cite[Lemma~7.2]{DD}, the natural map 
$\Omega^H_*(X) \xrightarrow{p^*} \Omega^{H \times G}_*(G \times X)$ is an 
isomorphism. On the other hand, the projection map $G \times X 
\to G \stackrel{H}{\times} X$ is $\left(H \times G\right)$-equivariant which 
is an $H$-torsor. Hence we get an isomorphism 
$\Omega^G_*\left(G \stackrel{H}{\times} X\right) \xrightarrow{\cong}
\Omega^{H \times G}_*\left(G \times X\right)$. The proposition follows
by combining these two isomorphisms. 
\end{proof}  

\section{Computations}\label{section:EXCOM}
Let $X$ be a $k$-scheme of dimension $d$ with a $G$-action. We have seen above
that unlike the situation of Chow groups, the cobordism group 
$\Omega_{i+l-g}\left(X \stackrel{G}{\times} U_j\right)$ is not independent of the
choice of the $l$-dimensional good pair $(V_j,U_j)$. This anomaly is rectified
by considering the quotients of the cobordism groups of the good pairs
by the niveau filtration. Our main result in this section is to show that if 
we suitably choose a sequence of good pairs ${\{(V_j, U_j)\}}_{j \ge 0}$, then 
the above equivariant cobordism group can be computed without taking 
quotients by the niveau filtration. This reduction is often very helpful in
computing the equivariant cobordism groups. 
  
\begin{thm}\label{thm:NO-Niveu}
Let ${\{(V_j, U_j)\}}_{j \ge 0}$ be a sequence of $l_j$-dimensional 
good pairs such that \\
$(i)$ $V_{j+1} = V_j \oplus W_j$ as representations of $G$ with ${\rm dim}(W_j) 
> 0$ and \\
$(ii)$ $U_j \oplus W_j \subset U_{j+1}$ as $G$-invariant open subsets. \\ 
Then for any scheme $X$ as above and any $i \in \Z$,
\[
\Omega^G_i(X) \xrightarrow{\cong} {\underset{j}\varprojlim} \
\Omega_{i+l_j-g}\left(X \stackrel{G}{\times} U_j\right).
\]
Moreover, such a sequence ${\{(V_j, U_j)\}}_{j \ge 0}$ of good pairs always
exists.
\end{thm}  
\begin{proof}
Let ${\{(V_j, U_j)\}}_{j \ge 0}$ be a sequence of good pairs as in the 
theorem. We have natural maps 
\[
\Omega_{i+l_{j+1}-g}\left(X \stackrel{G}{\times} U_{j+1}\right)
\to \Omega_{i+l_{j+1}-g}\left(X \stackrel{G}{\times} (U_j \oplus W_j)\right)
\xleftarrow{\cong} \Omega_{i+{l_j}-g}\left(X \stackrel{G}{\times} U_j\right),\]
where the first map is the restriction to an open subset and the second is
the pull-back via a vector bundle. Taking the quotients by the niveau
filtrations, we get natural maps ({\sl cf.} proof of Lemma~\ref{lem:ECob1})
\begin{equation}\label{eqn:NONIV*1}
\xymatrix@C.8pc{
\frac{\Omega_{i+l_{j+1}-g}\left(X \stackrel{G} {\times} U_{j+1}\right)}
{F_{d+l_{j+1}-g-j-1}\Omega_{i+l_{j+1}-g}\left(X \stackrel{G} {\times} U_{j+1}\right)}
\ar[r] &
\frac{\Omega_{i+l_{j+1}-g}\left(X \stackrel{G} {\times} (U_j \oplus W_j)\right)}
{F_{d+l_{j+1}-g-j-1}\Omega_{i+l_{j+1}-g}\left(X \stackrel{G} 
{\times} (U_j \oplus W_j)\right)} \\
\frac{\Omega_{i+l_j-g}\left(X \stackrel{G} {\times} U_j \right)}
{F_{d+l_j-g-j}\Omega_{i+l_j-g}\left(X \stackrel{G} {\times} U_j \right)} & 
\frac{\Omega_{i+l_j-g}\left(X \stackrel{G} {\times} U_j \right)}
{F_{d+l_j-g-j-1}\Omega_{i+l_j-g}\left(X \stackrel{G} {\times} U_j \right)}
\ar[u]_{\cong} \ar@{->>}[l]}
\end{equation}
where the right vertical arrow is an isomorphism by  Lemma~\ref{lem:Niv-Hi}.
Setting $X_j = X \stackrel{G}{\times} U_j$, we get natural maps
\begin{equation}\label{eqn:NONIV}
\xymatrix@C.6pc{
\Omega_{i+l_{j+1}-g}\left(X_{j+1} \right) \ar[r]^{\nu^{j+1}_j} \ar@{->>}[d] &
\Omega_{i+{l_j}-g}\left(X_j \right) \ar@{->>}[d] \\
\frac{\Omega_{i+l_{j+1}-g}\left(X_{j+1}\right)}
{F_{d+l_{j+1}-g-j-1}\Omega_{i+l_{j+1}-g}\left(X_{j+1} \right)} \ar[r] &
\frac{\Omega_{i+l_j-g}\left(X_{j} \right)}
{F_{d+l_j-g-j}\Omega_{i+l_j-g}\left(X_j \right)}.}
\end{equation}
Since $(V_j, U_j)$ is a good pair for each $j$, we see that
$\frac{\Omega_{i+l_j-g}\left(X_{j} \right)}
{F_{d+l_j-g-j}\Omega_{i+l_j-g}\left(X_j \right)} \cong {\Omega^G_i(X)}_j$.
Hence, we only have to show that the map 
\begin{equation}\label{eqn:NONIV*} 
{\underset{j}\varprojlim} \ \Omega_{i+{l_j}-g}\left(X_j \right) \to
{\underset{j}\varprojlim} \ \frac{\Omega_{i+l_j-g}\left(X_{j} \right)}
{F_{d+l_j-g-j}\Omega_{i+l_j-g}\left(X_j \right)}
\end{equation}
is an isomorphism in order to prove the theorem. 

To prove ~\eqref{eqn:NONIV*}, we only need to show that for any given 
$j \ge 0$, the map $\Omega_{i+l_{j'}-g}\left(X_{j'}\right) \xrightarrow{\nu^{j'}_j}
\Omega_{i+{l_j}-g}\left(X_j \right)$ factors through 
\begin{equation}\label{eqn:NONIV*2}
\frac{\Omega_{i+l_{j'}-g}\left(X_{j'}\right)}
{F_{d+l_{j'}-g-j'}\Omega_{i+l_{j'}-g}\left(X_{j'} \right)} \to
\Omega_{i+{l_j}-g}\left(X_j \right) \ \ {\rm for \ all} \  j' \gg j.
\end{equation} 
However, it follows from ~\eqref{eqn:NONIV*1} that $\nu^{j'}_j$
induces the map 
\[
\frac{\Omega_{i+l_{j'}-g}\left(X_{j'}\right)}
{F_{d+l_{j'}-g-j'}\Omega_{i+l_{j'}-g}\left(X_{j'} \right)} \to
\frac{\Omega_{i+l_{j}-g}\left(X_{j}\right)}
{F_{d+l_{j}-g-j'}\Omega_{i+l_{j}-g}\left(X_{j} \right)}.
\]
On the other hand $F_{d+l_{j}-g-j'}\Omega_{i+l_{j}-g}\left(X_{j} \right)$ vanishes
for $j' \gg j$. This proves ~\eqref{eqn:NONIV*2} and hence ~\eqref{eqn:NONIV*}.

Finally, it follows easily from the proof of Lemma~\ref{lem:ECob1} 
(see also \cite[Remark~1.4]{Totaro1}) that a sequence of 
good pairs as in Theorem~\ref{thm:NO-Niveu} always exists. 
\end{proof}

One consequence of Theorem~\ref{thm:NO-Niveu} is that for a linear algebraic
group $G$ acting on a scheme $X$ of dimension $d$, the forgetful map 
$r^G_X : \Omega^G_*(X) \to \Omega_*(X)$ ({\sl cf.} ~\eqref{eqn:resf})
can be easily shown to be analogous to the one used in 
\cite[Subsection~2.2]{EG} for the Chow groups. 
This interpretation of the forgetful map will have some interesting
applications in the computation of the non-equivariant cobordism using
the equivariant techniques ({\sl cf.} \cite{Krishna5}, \cite{Krishna3}).

So let $\{(V_j, U_j)\}_{j \ge 0}$
be a sequence of good pairs as in Theorem~\ref{thm:NO-Niveu}. We 
choose a $k$-rational point $x \in U_0$ and let $x_j$ be its image
in ${U_j}/G$ under the natural map $U_0 \to {U_0}/G \to {U_{j}}/G$. 
Setting $X_j = X \stackrel{G}{\times}U_j$, this yields a commutative diagram
\begin{equation}\label{eqn:resf1}
\xymatrix@C.7pc{
X \times U_j \ar[r]^{\ \ \ p_j} \ar[d]_{\pi_j} &
X_j \ar[r]^{\phi_j} \ar[d]_{\psi_j} & X_{j+1} \ar[d]^{\psi_{j+1}} \\
U_j \ar[r] & {U_j}/G \ar[r] & {U_{j+1}}/G}
\end{equation}
such that the left square is Cartesian and 
\begin{equation}\label{eqn:resf2}
X \cong \pi^{-1}(x) \xrightarrow{\cong}  \phi^{-1}_j(x_j) \xrightarrow{\cong} 
\phi^{-1}_{j+1}(x_{j+1}).
\end{equation}
Let $\nu_j : \phi^{-1}_j(x_j) \inj X_j$ be the closed embedding.
Notice that since ${U_j}/G$ is smooth and $\psi_j$ is flat, it follows
that $\nu_j$ is a regular embedding. Using the identification in 
~\eqref{eqn:resf2}, we get maps $\nu^*_j : \Omega_*(X_j) \to 
\Omega_*(X)$ such that $\nu^*_j \circ \phi^*_j = \nu^*_{j+1}$. Taking the limit 
over $j \ge 0$, this yields for any $i \in \Z$, a restriction map
$\wt{r}^G_X : \Omega^G_i(X) = {\underset{j \ge 0}\varprojlim} \
\Omega^G_i(X)_j \to \Omega_i(X)$. Since any two rational points in
${U_j}/G$ define the same class in $\Omega_0\left({U_j}/G\right)$ 
by \cite[Lemma~4.5.10]{LM}, we see that $\wt{r}^G_X$ is well-defined.

\begin{cor}\label{cor:TWR}
The maps $r^G_X, \ \wt{r}^G_X : \Omega^G_*(X) \to \Omega_*(X)$ coincide.
\end{cor}
\begin{proof} 
Using the construction of the map $r^G_X$ in ~\eqref{eqn:resf} and the
diagram ~\eqref{eqn:resf1}, it suffices to show that the natural maps
\[
\frac{\Omega_*(X)}{F_{d-j}\Omega_*(X)} \leftarrow 
\frac{\Omega_*\left(X \times V_j\right)} 
{F_{d+l_j -j}\Omega_*\left(X\times V_j\right)} \to
\frac{\Omega_*\left(X \times U_j\right)}
{F_{d+l_j -j}\Omega_*\left(X\times U_j\right)}
\]
are isomorphisms for all $j \gg 0$, where the two maps are
the restrictions to the zero-section and the open subset.
But this follows immediately from Corollary~\ref{cor:Niv-open} and
Lemma~\ref{lem:Niv-Hi}.
\end{proof}


\subsection{Formal group law in equivariant cobordism}
\label{subsection:FGL*}
Let $G$ be a linear algebraic group over $k$ acting on a scheme $X$ of 
dimension $d$. Let ${\{(V_j, U_j)\}}_{j \ge 0}$ be a sequence of $l_j$-dimensional 
good pairs as in Theorem~\ref{thm:NO-Niveu}. Letting $X_j = 
X \stackrel{G}{\times} U_j$, we see that for every $j \ge 0$,
$\Omega_*\left(X_j\right) = {\underset{i \in \Z}\bigoplus} \
\Omega_{i+l_j-g}(X_j)$ is an $\bL$-module and for $j' \ge j$, there is
a natural surjection $\Omega_*\left(X_{j'}\right) \surj
\Omega_*\left(X_j\right)$ of $\bL$-modules.

Given $G$-equivariant line bundles $L, M$ on $X$, we get line bundles
$L_j, M_j$ on $X_j$, where $L_j = L_j \stackrel{G}{\times} U_j$ for
$j \ge 0$. The formal group law of the non-equivariant cobordism yields
\[
c_1\left((L \otimes M)_j\right)
=  c_1\left(L_j  \otimes M_j\right) 
= c_1(L_j) + c_1(M_j) + {\underset{i,i' \ge 1}\sum} \
a_{i,i'} \left(c_1(L_j)\right)^i \circ \left(c_1(M_j)\right)^{i'}.
\]
Note that if $(x_j) \in \Omega^G_i(X)$, then the evaluation of the operator
$c^T_1(L)(x_j)$ at any level $j \ge 0$ is a finite sum above.

Taking the limit over $j \ge 0$ and noting that the sum (and the product) in the
equivariant cobordism groups are obtained by taking the limit of the
sums (and the products) at each level of the inverse system, we get
the same formal group law for the equivariant Chern classes:
\begin{equation}\label{eqn:EFGL}
c^G_1(L \otimes M) = c^T_1(L) + c^T_1(M) +
{\underset{i,i' \ge 1}\sum} \
a_{i,i'} \left(c^T_1(L)\right)^i \circ \left(c^T_1(M)\right)^{i'}.
\end{equation}
One should also observe that unlike the case of ordinary cobordism,
the evaluation of the above sum on any given equivariant cobordism cycle 
may no longer be finite. In other words, the equivariant Chern classes
are not in general locally nilpotent.

\subsection{Cobordism ring of classifying spaces}
Let $R$ be a Noetherian ring and let $A = {\underset{j \in \Z}\oplus} A_j$ be 
a $\Z$-graded $R$-algebra with $R \subset A_0$. Recall that the {\sl 
graded power series ring} $S^{(n)} = {\underset{i \in \Z} \oplus} S_i$ is a
graded ring such that $S_i$ is the set of formal power series of the form 
$f({\bf t}) = {\underset{m({\bf t}) \in \sC} \sum} a_{m({\bf t})} m({\bf t})$
such that $a_{m({\bf t})}$ is a homogeneous element of $A$ of degree 
$|a_{m({\bf t})}| $and $|a_{m({\bf t})}| + |m({\bf t})| = i$. 
Here, $\sC$ is the set of all monomials in 
${\bf t} = (t_1, \cdots , t_n)$ and $|m({\bf t})| = i_1 + \cdots + i_n$ if 
$m({\bf t}) = t^{i_1}_1 \cdots t^{i_n}_n$.
We call $|m({\bf t})|$ to be the degree of the monomial $m({\bf t})$.

We shall often write the above graded power series ring as
${A[[{\bf t}]]}_{\rm gr}$ to distinguish it from the
usual formal power series ring. 
Notice that if $A$ is only non-negatively graded, then $S^{(n)}$ is nothing but 
the standard 
polynomial ring $A[t_1, \cdots , t_n]$ over $A$. It is also easy to see that 
$S^{(n)}$ is indeed a graded ring which is a subring of the formal power series 
ring $A[[t_1, \cdots , t_n]]$. The following result summarizes some basic 
properties of these rings. The proof is straightforward and is left as an 
exercise.

\begin{lem}\label{lem:GPSR}
$(i)$ There are inclusions of rings $A[t_1, \cdots , t_n] \subset S^{(n)} \subset
A[[t_1, \cdots , t_n]]$, where the first is an inclusion of graded rings. \\
$(ii)$ These inclusions are analytic isomorphisms with respect to the
${\bf t}$-adic topology. In particular, the induced maps of the associated
graded rings
\[
A[t_1, \cdots , t_n] \to {\rm Gr}_{({\bf t})} S^n \to 
{\rm Gr}_{({\bf t})} A[[t_1, \cdots , t_n]]
\]
are isomorphisms. \\
$(iii)$ $S^{(n-1)}[[t_i]]_{\rm gr} \xrightarrow{\cong} S^{(n)}$. \\
$(iv)$ $\frac{S^{(n)}}{(t_{i_1}, \cdots , t_{i_r})} \xrightarrow{\cong} S^{(n-r)}$ 
for any $n \ge r \ge 1$, where $S^{(0)} = A$. \\ 
$(v)$ The sequence $\{t_1, \cdots , t_n\}$ is a regular sequence in $S^{(n)}$.
\\
$(vi)$ If $A = R[x_1, x_2, \cdots ]$ is a polynomial ring with
$|x_i| < 0$ and ${\underset{i \to \infty}{\rm lim}} \ |x_i| = - \infty$, then
$S^{(n)} \xrightarrow{\cong}
{\underset{i}\varprojlim} \ {R[x_1, \cdots , x_i][[{\bf t}]]}_{\rm gr}$.
\end{lem}
 
Since we shall mostly be dealing with the graded power series ring in this 
text, we make the convention of writing 
${A[[{\bf t}]]}_{\rm gr}$ as 
$A[[{\bf t}]]$, while the standard formal
power series ring will be written as $\widehat{A[[{\bf t}]]}$. 

\begin{exms}\label{exms:BT}
In the following examples, we compute $\Omega^*(BG) = \Omega^*_G(k)$ for some 
classical groups $G$ over $k$. These computations follow directly from the 
definition of equivariant cobordism and suitable choices of good pairs.
We first consider the case when $G = \G_m$ is the
multiplicative group. For any $j \ge 1$, we choose the good pair $(V_j, U_j)$,
where $V_j$ is the $j$-dimensional representation of $\G_m$ with all weights
$-1$ and $U_j$ is the complement of the origin. We see then that ${U_j}/{\G_m}
\cong \P^{j-1}_k$. Let $\zeta$ be the class of ${c_1}(\sO(-1))(1) \in
\Omega^1(\P^{j-1}_k)$. The projective bundle formula for the ordinary algebraic
cobordism implies that ${(\Omega^i_G)}_j = 
{\underset{0 \le p \le j-1} \bigoplus} \bL^{i-p}\zeta^p$. Taking the inverse
limit over $j \ge 1$, we find from this that for $i \in \Z$, 
\[
\Omega^i_{\G_m}(k) = {\underset{p \ge 0} \prod} \bL^{i-p}\zeta^p.
\]
It particular, it is easy to see that there is a natural map
\begin{equation}\label{eqn:BT1}
\Omega^*(B\G_m) \to \bL[[t]]
\end{equation}
\[
\left(x_{i_1} = \prod a^{i_1}_p \zeta^p, \cdots ,
x_{i_n} = \prod a^{i_n}_p \zeta^p\right) \mapsto {\underset{p \ge 0}\sum}
\left({\underset{1 \le j \le n}\sum} a^{i_j}_p \right) t^p,
\]
which is an isomorphism of graded $\bL$-algebras.

For a general split torus $T$ of rank $n$, we choose a basis $\{\chi_1,
\cdots , \chi_n\}$ of the character group $\widehat{T}$. This is equivalent
to a decomposition $T = T_1 \times  \cdots  \times T_n$ with each $T_i$
isomorphic to $\G_m$ and $\chi_i$ is a generator of $\widehat{T_i}$.
Let $L_{\chi}$ be the one-dimensional representation of $T$, where $T$ acts via 
$\chi$. For any $j \ge 1$, we take the good pair $(V_j, U_j)$ such that
$V_j = \stackrel{n}{\underset{i = 1} \prod} 
L^{\oplus j}_{\chi_i}$, $U_j = \stackrel{n}{\underset{i = 1} \prod} 
\left(L^{\oplus j}_{\chi_i} \setminus \{0\}\right)$ and 
$T$ acts on $V_j$ by $(t_1, \cdots , t_n)(x_1, \cdots , x_n)
= \left(\chi_1(t_1)(x_1), \cdots , \chi_n(t_n)(x_n)\right)$. 
It is then easy to see that
${U_j}/{T} \cong X_1 \times \cdots \times X_n$ with each $X_i$ isomorphic to 
$\P^{j-1}_k$.
Moreover, the $T$-line bundle $L_{\chi_i}$ gives the line bundle 
$L_{\chi_i} \stackrel{T_i}{\times} \left(L^{\oplus j}_{\chi_i} \setminus \{0\}\right)
\to X_i$ which is $\sO(\pm 1)$. Letting $\zeta_i$ be the first
Chern class of this line bundle, the projective bundle formula for the
non-equivariant cobordism shows that
\[
\Omega^i_{T}(k) = {\underset{p_1, \cdots , p_n  \ge 0} \prod} 
\bL^{i-(\stackrel{n}{\underset{i=1} \sum} p_i)}\zeta^{p_1}_1 \cdots \zeta^{p_n}_n,
\]
which is isomorphic to the set of formal power series in 
$\{\zeta_1, \cdots , \zeta_n\}$ of degree $i$ with coefficients in 
$\bL$. It particular, one concludes as in the rank one case above that

\begin{prop}\label{prop:TorusC}
Let $\{\chi_1, \cdots , \chi_n\}$ be a chosen basis of the character group of a 
split torus $T$ of rank $n$. The assignment $t_i \mapsto c^T_1(L_{\chi_i})$
yields a graded $\bL$-algebra isomorphism
\[
\bL[[t_1, \cdots , t_n]] \to \Omega^*(BT).
\]
\end{prop}

For $G = GL_n$, we can take a good pair for $j$ to be $(V_j, U_j)$, where
$V_j$ is the vector space of $n \times p$ matrices with $p > n$ with $GL_n$
acting by left multiplication, and $U_j$ is the open subset of matrices of
maximal rank. Then the mixed quotient is the Grassmannian $Gr(n,p)$. 
We can now calculate the cobordism ring of $Gr(n,p)$ using the projective
bundle formula (by standard stratification technique) and then we can use
the similar calculations as above to get a natural isomorphism
\begin{equation}\label{eqn:BT2}
\Omega^*(BGL_n) \to \bL[[\gamma_1, \cdots , \gamma_n]]
\end{equation}
of graded $\bL$-algebras, where $\gamma_i$'s  are the elementary symmetric
polynomials in $t_1, \cdots , t_n$ that occur in ~\eqref{eqn:BT1}.
Furthermore, using the fact that ${U_j}/{SL_n}$ fibers over $Gr(n,p)$ as
the complement of the zero-section of the determinant bundle of the 
tautological rank $n$ bundle on $Gr(n,p)$ and using the localization 
sequence of Theorem~\ref{thm:Levine-M}, we see that 
$\Omega^*(BSL_n) \xrightarrow{\cong} \bL[[\gamma_2, \cdots , \gamma_n]]$.
\end{exms}

\begin{remk}\label{remk:Non-Obv}
The cobordism rings of $BGL_n$ and $BSL_n$ have also been calculated in
\cite[Section~4]{DD} by an indirect method which involves the comparison
of these groups with the known complex cobordism rings of the corresponding
topological classifying spaces. 
However, the comparison result in {\sl loc. cit.} assumes
the existence of a natural ring homomorphism from the algebraic to the complex 
cobordism of classifying spaces, which is not immediately obvious.
\end{remk}

\section{Comparison with other cohomology theories}
\label{section:COMP}
In this paper, we fix the following notation for the tensor product
while dealing with inverse systems of modules over a commutative ring.
Let $A$ be a commutative ring with unit and let $\{L_n\}$ and $\{M_n\}$
be two inverse systems of $A$-modules with inverse limits $L$ and $M$
respectively. Following \cite{Totaro2}, one defines the 
{\sl topological tensor product} of $L$ and $M$ by
\begin{equation}\label{eqn:TTP}
L \widehat{\otimes}_A M : = 
{\underset{n}\varprojlim} \ (L_n {\otimes}_A M_n).
\end{equation} 
In particular, if $D$ is an integral domain with quotient field $F$ and if
$\{A_n\}$ is an inverse system of $D$-modules with inverse limit $A$, one 
has $A \widehat{\otimes}_{D} F = 
{\underset{n}\varprojlim} \ (A_n {\otimes}_{D} F)$.
The examples $\widehat{\Z_{(p)}} = {\underset{n}\varprojlim} \ {\Z}/{p^n}$ and 
$\Z[[x]] \otimes_{\Z} \Q \to {\underset{n}\varprojlim} \ 
\frac{\Z[x]}{(x^n)} {\otimes}_{\Z} \Q = \Q[[x]]$ show that the map  
$A \otimes_D F \to A \widehat{\otimes}_{D} F$ is in
general neither injective nor surjective. 

If $R$ is a $\Z$-graded ring and if $M$ and $N$ are two $R$-graded modules, 
then recall that $M\otimes_R N$ is also a graded $R$-module given by the 
quotient of $M \otimes_{R_0} N$ by the graded submodule generated by the 
homogeneous elements of the type $ax \otimes y - x \otimes ay$ where $a, x$ 
and $y$ are the homogeneous elements of $R$, $M$ and $N$ respectively.
If all the graded pieces $M_i$ and $N_i$ are the limits of inverse systems
$\{M^{\lambda}_i\}$ and $\{N^{\lambda}_i\}$ of $R_0$-modules, we define the
{\sl graded topological tensor product} as 
$M\widehat{\otimes}_R N = {\underset{i \in \Z}\bigoplus} 
\left(M\widehat{\otimes}_R N\right)_i$, where
\begin{equation}\label{eqn:TTP}
\left(M\widehat{\otimes}_R N\right)_i =
{\underset{\lambda}\varprojlim}
\left({\underset{j + j' = i}\bigoplus} \ 
\frac{M^{\lambda}_j \otimes_{R_0} N^{\lambda}_{j'}}
{\left(ax \otimes y - x \otimes ay\right)}\right).
\end{equation}
Notice that this reduces to the ordinary tensor product of graded $R$-modules
if the underlying inverse systems are trivial. 

\subsection{Comparison with equivariant Chow groups}
\label{subsection:Com-Chow}
Let $X$ be a $k$-scheme of dimension $d$ with a $G$-action.
It was shown by Levine and Morel \cite{LM} that there is a natural map
$\Omega_*(X) \to CH_*(X)$ of graded abelian groups which is a ring homomorphism
if $X$ is smooth. Moreover, this map induces a graded isomorphism
\begin{equation}\label{eqn:Cob*-Chow}
\Omega_*(X) \otimes_{\bL} \Z \xrightarrow{\cong} CH_*(X).
\end{equation}

Recall from \cite{Totaro1} and \cite{EG} that the equivariant Chow groups of
$X$ are defined as $CH^G_i(X) = CH_{i+l-g}\left(X \stackrel{G}{\times} U\right)$,
where $(V, U)$ is an $l$-dimensional good pair corresponding to $d-i$. It 
is known that $CH^G_i(X)$ is well-defined and can be non-zero for any
$- \infty < i \le d$. We set $CH^G_*(X) = {\underset{i}\bigoplus} CH^G_i(X)$.
If $X$ is equi-dimensional, we let $CH^i_G(X) = CH^G_{d-i}(X)$ and set
$CH^*_G(X) = {\underset{i \ge 0}\bigoplus} CH_G^i(X)$. Notice that in this case,
$CH^i_G(X)$ is same as $CH^i\left(X \stackrel{G}{\times} U\right)$,
where $(V, U)$ is an $l$-dimensional good pair corresponding to $i$.

If we fix $i \in \Z$ and choose an $l$-dimensional good pair $(V_j, U_j)$
corresponding to $d-j \ge {\rm max} (0, d-i+1)$, the universality of the 
algebraic cobordism gives a unique map 
$\Omega_{i+l-g}\left(X \stackrel{G}{\times} U_j\right) \to
CH_{i+l-g}\left(X \stackrel{G}{\times} U_j\right)$. By Lemma~\ref{lem:Niv-Chow},
this map factors through
\begin{equation}\label{eqn:Cob*-Chow1}
\frac{\Omega_{i+l-g}\left(X \stackrel{G}{\times} U_j\right)}
{F_{i+l-g-1}\Omega_{i+l-g}\left(X \stackrel{G}{\times} U_j\right)} \to
CH_{i+l-g}\left(X \stackrel{G}{\times} U_j\right).
\end{equation}
Since $j \le i-1$ by the choice, we have $d+l-g -(d-j) \le i+l-g-1$ and hence
we get the map
\begin{equation}\label{eqn:Cob*-Chow2}
{\Omega^G_{i+l-g}(X)}_{d-j} = 
\frac{\Omega_{i+l-g}\left(X \stackrel{G}{\times} U_j\right)}
{F_{d+l-g-(d-j)}\Omega_{i+l-g}\left(X \stackrel{G}{\times} U_j\right)} \to
CH_{i+l-g}\left(X \stackrel{G}{\times} U_j\right).
\end{equation}
It is easily shown using the proof of Lemma~\ref{lem:ECob1} that this map
is independent of the choice of the good pair $(V_j, U_j)$. Taking the 
inverse limit over $d-j \ge 0$, we get a natural map $\Omega^G_i(X) \to 
CH^G_i(X)$ and hence a map of graded abelian groups
\begin{equation}\label{eqn:Cob*-Chow3}
\Phi_X : \Omega^G_*(X) \to CH^G_*(X)
\end{equation}
which is in fact a map of graded $\bL$-modules.
Notice that the right side of ~\eqref{eqn:Cob*-Chow2} does not depend on $j$
as long as $d-j \gg 0$. If $X$ is equi-dimensional, we write the above map
cohomologically as $\Omega^*_G(X) \to CH^*_G(X)$.

\begin{prop}\label{prop:CBCH} 
The map $\Phi_X$ induces an isomorphism of graded $\bL$-modules
\[
\Phi_X : \Omega^G_*(X) \widehat{\otimes}_{\bL} \Z \xrightarrow{\cong} CH^G_*(X).
\]
\end{prop}
\begin{proof} 
We fix $i \in \Z$ and choose an 
$l$-dimensional good pair $(V_j, U_j)$ corresponding to 
$d-j \ge {\rm max} (0, d-i+1)$ and set $X_j = X \stackrel{G}{\times} U_j$.
This gives the natural map as in ~\eqref{eqn:Cob*-Chow2} which in
turn yields an exact sequence 
\begin{equation}\label{eqn:Cob*-Chow2*}
0 \to \left(\bL^{<0} \frac{\Omega_*(X_j)}{F_{l+j-g}\Omega_*(X_j)} \cap 
\frac{\Omega_{i+l-g}(X_j)}{F_{l+j-g}\Omega_*(X_j)} \right) \to
\frac{\Omega_{i+l-g}(X_j)}{F_{l+j-g}\Omega_{i+l-g}(X_j)} \to CH_{i+l-g}(X_j) \to 0
\end{equation}
by ~\eqref{eqn:Cob*-Chow}.
Since $\frac{\Omega_*(X_{j-1})}{F_{l+j-1-g}\Omega_*(X_{j-1})} \to
\frac{\Omega_*(X_j)}{F_{l+j-g}\Omega_*(X_j)}$ is a surjective map of
graded $\bL$-modules (as can be
easily seen by choosing a good pair corresponding to $d-j+1$), we see that 
the left and the middle terms of the above exact sequence form inverse systems
of surjective maps. Taking the limit, we get an exact sequence
\begin{equation}\label{eqn:Cob*-Chow2*1}
0 \to {\underset{d-j \ge 0}\varprojlim}
\left(\bL^{<0} \frac{\Omega_*(X_j)}{F_{l+j-g}\Omega_*(X_j)} \cap 
\frac{\Omega_{i+l-g}(X_j)}{F_{l+j-g}\Omega_{i+l-g}(X_j)} \right) \to \Omega^G_i(X) \to
CH^G_i(X) \to 0.
\end{equation}
Taking the direct sum over $i \in \Z$ and using ~\eqref{eqn:TTP}, we get
\[
\Omega^G_*(X) \widehat{\otimes}_{\bL} \Z \xrightarrow{\cong} CH^G_*(X).
\]
\end{proof}
Let $C(G) = CH^*_G(k)$ denote the equivariant Chow ring of the field $k$.
\begin{cor}\label{cor:CCH*}
For a $k$-scheme $X$ with a $G$-action, the natural map
\[
\Omega^G_*(X) {\otimes}_{S(G)} C(G) \to CH^G_*(X)
\]
is an isomorphism of $C(G)$-modules. This is a ring isomorphism if $X$
is smooth.
\end{cor}
\begin{proof} 
It is clear that the above map is a ring homomorphism if $X$ is smooth.
So we only need to prove the first assertion. But this follows directly 
from the isomorphisms $\Omega^G_*(X) {\otimes}_{S(G)} C(G) 
\cong \Omega^G_*(X) {\otimes}_{S(G)} \left(S(G) \widehat{\otimes}_{\bL} \Z
\right) \cong \Omega^G_*(X)  \widehat{\otimes}_{\bL} \Z$ using
Proposition~\ref{prop:CBCH}.
\end{proof}

\subsection{Comparison with equivariant $K$-theory}
It was shown by Levine and Morel in \cite[Corollary~11.11]{LM1} that the
universal property of the algebraic cobordism implies that there is a
canonical isomorphism of oriented cohomology theories
\begin{equation}\label{eqn:COb-K} 
\Omega^*(X) \otimes_{\bL} \Z[\beta, \beta^{-1}] \xrightarrow{\cong} 
K_0(X)[\beta, \beta^{-1}]
\end{equation}
in the category of smooth $k$-schemes.
This was later generalized to a complete algebraic analogue of the
Conner-Floyd isomorphism
\[
MGL^* \otimes_{\bL} \Z[\beta, \beta^{-1}] \xrightarrow{\cong} 
K_*(X)[\beta, \beta^{-1}]
\]
between the motivic cobordism and algebraic $K$-theory by Panin, Pimenov and 
R\"ondigs \cite{PPR}.
Since the equivariant cobordism is a Borel style cohomology theory, one
can not expect an equivariant version of the isomorphism ~\eqref{eqn:COb-K}
even with the rational coefficients. However, we show here that the
equivariant Conner-Floyd isomorphism holds after we base change the above
by the completion of the representation ring of $G$ with respect to the
ideal of virtual representations of rank zero. In fact, it can be shown 
easily that such a base change is the minimal requirement. In 
Theorem~\ref{thm:CKtheory}, all cohomology groups are considered with
rational coefficients ({\sl cf.} Section~\ref{section:tori}).

For a linear algebraic group $G$, let $R(G)$ denote the representation ring
of $G$. Let $I$ denote the ideal of of virtual representations of rank zero
in $R(G)$ and let $\widehat{R(G)}$ denote the associated completion of
$R(G)$. Let $\widehat{C(G)}$ denote the completion of $C(G)$ 
with respect to the augmentation ideal of algebraic cycles of
positive codimensions.
For a scheme $X$ with $G$-action, let $K^G_0(X)$ denote the Grothendieck
group of $G$-equivariant vector bundles on $X$. By \cite[Theorem~4.1]{EDR},
there is a natural ring isomorphism $\widehat{R(G)} \xrightarrow{\cong}
\widehat{C(G)}$ given by the equivariant Chern character. We identify these 
two rings via this isomorphism.
In particular, the maps $S(G) \surj C(G) \to \widehat{C(G)}$ yield
a ring homomorphism $S(G) \to \widehat{R(G)}$.

\begin{thm}\label{thm:CKtheory}
For a smooth scheme $X$ with a $G$-action, there is a natural isomorphism
of rings
\[
\Psi_X: \Omega^*_G(X) {\otimes}_{S(G)} \widehat{R(G)} \xrightarrow{\cong}
K^G_0(X) \otimes_{R(G)} \widehat{R(G)}.
\]
\end{thm}
\begin{proof}
By \cite[Theorem~1.2]{Krishna}, there is a Chern character isomorphism
$K^G_0(X) \otimes_{R(G)} \widehat{R(G)} \xrightarrow{\cong}
CH^*(X) \otimes_{C(G)} \widehat{C(G)}$
of cohomology rings. Thus, we only need to show that the map
$\Omega^*_G(X) {\otimes}_{S(G)} \widehat{C(G)} \to CH^*(X) \otimes_{C(G)} 
\widehat{C(G)}$
is an isomorphism. However, we have
\[
\begin{array}{lll}
\Omega^*_G(X) {\otimes}_{S(G)} \widehat{C(G)} & \cong &
\left(\Omega^*_G(X) {\otimes}_{S(G)} C(G)\right) \otimes_{C(G)} 
\widehat{C(G)} \\
& \cong & CH^*_G(X) \otimes_{C(G)} 
\widehat{C(G)},
\end{array}
\]
where the last isomorphism follows from Corollary~\ref{cor:CCH*}.
This finishes the proof.
\end{proof}

\subsection{Comparison with complex cobordism}
Let $G$ be a complex Lie group acting on a finite $CW$-complex $X$. We define 
the {\sl equivariant complex cobordism} ring of $X$ as 
\begin{equation}\label{eqn:ECompC}
MU^*_G(X) := MU^*\left(X \stackrel{G}{\times} EG\right)
\end{equation}
where $EG \to BG$ is universal principal $G$-bundle over the classifying
space $BG$ of $G$. If $E'G \to B'G$ is another such bundle, then the
projection $(X \times EG \times E'G)/G \to X \stackrel{G}{\times} EG$ is a
fibration with contractible fiber. In particular, $MU^*_G(X)$ is 
well-defined. Moreover, if $G$ acts freely on $X$ with quotient $X/G$, then
the map $X \stackrel{G}{\times} EG \to X/G$ is a fibration with contractible
fiber $EG$ and hence we get $MU^*_G(X) \cong MU^*(X/G)$. 

For a linear algebraic group $G$ over $\C$ acting on a $\C$-scheme $X$, let
$H^*_G(X, A)$ denote the (equivariant) cohomology of the complex analytic 
space $X(\C)$ with coefficients in the ring $A$.

\begin{prop}\label{prop:Alg-Comp}
Assume that $X \in \sV^S_G$ is such that $H^*_G(X, \Z)$ is torsion-free. Then 
there is a natural homomorphism of graded rings
\[
\rho^G_X : \ \Omega^*_G(X) \to MU^{2*}_G(X).
\]
\end{prop}
\begin{proof} If $\{(V_j, U_j)\}$ is a sequence of good pairs as in 
Theorem~\ref{thm:NO-Niveu}, then the universality of the Levine-Morel
cobordism gives a natural $\bL$-algebra map of inverse systems
\[
\Omega^i\left(X \stackrel{G}{\times} U_j\right) \to
MU^{2i}\left(X \stackrel{G}{\times} U_j\right)
\]
which after taking limits yields the map
\begin{equation}\label{eqn:Alg-Comp1}
\Omega^i_G(X) = {\underset{j \ge 0}\varprojlim} \
\Omega^i\left(X \stackrel{G}{\times} U_j\right) \to
{\underset{j \ge 0}\varprojlim} \
MU^{2i}\left(X \stackrel{G}{\times} U_j\right)
\cong {\underset{j \ge 0}\varprojlim} \
MU^{2i}\left(X \stackrel{G}{\times} (BG)_j\right).
\end{equation}
Thus, we only need to show that the map 
\[
MU^i\left(X \stackrel{G}{\times} EG\right) \to
{\underset{j \ge 0}\varprojlim} \
MU^{i}\left(X \stackrel{G}{\times} (BG)_j\right)
\]
is an isomorphism. But this follows from our assumption, 
\cite[Corollary~1]{Landweber} and the Milnor exact sequence
\begin{equation}\label{eqn:Alg-Comp2}
\xymatrix@C.6pc{
0 \ar[r] & {\underset{j \ge 0}{\varprojlim}^1} \
MU^{i}\left(X \stackrel{G}{\times} (BG)_j\right)  \ar[r] & 
MU^i\left(X \stackrel{G}{\times} EG\right)  \ar[r] & 
{\underset{j \ge 0}\varprojlim} \
MU^{i}\left(X \stackrel{G}{\times} (BG)_j\right)  \ar[r] &  0.}
\end{equation}
\end{proof}
The following is an immediate consequence of Propositions~\ref{prop:CBCH} 
and ~\ref{prop:Alg-Comp}.
\begin{cor}\label{cor:Alg-Comp*1}
For any $X \in \sV^S_G$, there is a natural map of graded 
$\bL_{\Q}$-algebras
\[
\rho^G_X : \ \Omega^*_G(X)_{\Q} \to MU^{2*}_G(X)_{\Q}.
\]
In particular, there is a natural ring homomorphism
\[
\ov{\rho}^G_X : \ CH^*_G(X)_{\Q} \to MU^{2*}_G(X)_{\Q} {\widehat{\otimes}}_{\bL_{\Q}} 
\Q
\]
which factors the cycle class map $CH^*_G(X) \to H^{2*}_G(X, \Q)$.
\end{cor}
\begin{cor}\label{cor:Alg-Comp*}
There is a natural morphism $\Omega^*(BG) \to MU^{2*}(BG)$ of graded 
$\bL$-algebras. In particular, there is a natural ring homomorphism
$CH^*(BG) \to MU^{2*}(BG) {\widehat{\otimes}}_{\bL} \Z$ which factors the
cycle class map $CH^*(BG) \to H^{2*}(BG, \Z)$.
\end{cor}
\begin{proof}
The first assertion follows immediately from ~\eqref{eqn:Alg-Comp1}, 
~\eqref{eqn:Alg-Comp2} and \cite[Theorem~1]{Landweber} using the fact that 
$BG$ is homotopy equivalent to the classifying space of its maximal compact 
subgroup. The second assertion follows from the first and 
Proposition~\ref{prop:CBCH} using the identification $\bL \xrightarrow{\cong}
MU^*$. 
\end{proof} 
\begin{remk}\label{remk:TOT}
The map $CH^*(BG) \to MU^{2*}(BG) {\widehat{\otimes}}_{\bL} \Z$ has also been
constructed by Totaro \cite{Totaro1} by a more direct method.
\end{remk}
We shall study the above realization maps in more detail in the next section.

\section{Reduction of arbitrary groups to tori}\label{section:tori}
The main result of this section is to show that with the rational coefficients,
the equivariant cobordism of schemes with an action of a connected linear
algebraic group can be written in terms of the Weyl group invariants of the
equivariant cobordism for the action of the maximal torus. This reduces
the problems about the equivariant cobordism to the case where the underlying
group is a torus. We draw some consequences of this for the cycle class
map from the rational Chow groups to the complex cobordism groups of 
classifying spaces. We first prove some reduction results about the equivariant
cobordism which reflect the relations between the $G$-equivariant cobordism
and the equivariant cobordism for actions of subgroups of $G$. The results 
of this section are used in \cite{Krishna4} and \cite{Krishna3} to compute 
the non-equivariant cobordism ring of flag varieties and flag bundles.
\\
\\
\noindent
{\bf Notation:} All results in this section will be proven with
the rational coefficients. In order to simplify our notations, an abelian
group $A$ from now on will actually mean the $\Q$-vector space 
$A \otimes_{\Z} \Q$, and an inverse limit of abelian groups will mean the
limit of the associated $\Q$-vector spaces. In particular, all cohomology groups
will be considered with the rational coefficients and $\Omega^G_i(X)$ will
mean 
\[
\Omega^G_i(X) : = {\underset{j}\varprojlim} \ \left({\Omega^G_i(X)}_j 
\otimes_{\Z} \Q\right). 
\]
Notice that this is same as $\Omega^G_i(X) {\widehat{\otimes}}_{\Z} \Q$ in our
earlier notation.

\begin{prop}\label{prop:red-torus}
Let $G$ be a connected and reductive group over $k$. Let $B$ be a Borel 
subgroup of $G$ containing a maximal torus $T$ over $k$. Then for any
$X \in \sV_G$, the restriction map
\begin{equation}\label{eqn:Borel2}
\Omega^B_*\left(X\right) \xrightarrow{r^B_{T,X}} \Omega^T_*\left(X\right),
\end{equation}
is an isomorphism.
\end{prop}
\begin{proof}
By Proposition~\ref{prop:Morita}, we only need to show that
\begin{equation}\label{eqn:Borel1}
\Omega_*^B\left(B \stackrel{T}{\times} X\right) \cong
\Omega_*^B\left(X\right).
\end{equation}
By \cite[XXII, 5.9.5]{SGA3}, there exists a characteristic
filtration $B^u = U_0 \supseteq U_1 \supseteq \cdots \supseteq U_n =
\{1\}$ of the unipotent radical $B^u$ of $B$ such that ${U_{i-1}}/{U_i}$
is a vector group, each $U_i$ is normal in $B$ and $TU_i = T \ltimes U_i$. 
Moreover, this filtration also implies that for each $i$, the natural map
$B/{BU_i} \to B/{TU_{i-1}}$ is a torsor under the vector bundle
${U_{i-1}}/{U_i} \times B/{TU_{i-1}}$ on $B/{TU_{i-1}}$. Hence, the 
homotopy invariance ({\sl cf.} Theorem~\ref{thm:Basic}) gives an isomorphism
\[
\Omega_*^B\left(B/{TU_{i-1}} \times X\right) \xrightarrow{\cong}
\Omega_*^B\left(B/{TU_i} \times X\right).
\]
Composing these isomorphisms successively for $i = 1, \cdots ,n$, we get
\[
\Omega_*^B\left(X\right) \xrightarrow{\cong}
\Omega_*^B\left(B/T \times X\right).
\]
The canonical isomorphism of $B$-varieties $B \stackrel{T}{\times} X \cong
B/T \times X$ ({\sl cf.} Proposition~\ref{prop:Morita}) 
now proves ~\eqref{eqn:Borel1} and hence ~\eqref{eqn:Borel2}.
\end{proof}
\begin{prop}\label{prop:NRL}
Let $H$ be a possibly non-reductive group over $k$.
Let $H = L \ltimes H^u$ be the Levi decomposition of $H$ 
(which exists since $k$ is of characteristic zero).
Then the restriction map 
\begin{equation}\label{eqn:NRL0}
{\Omega^H_*\left(X\right)} \xrightarrow{r^H_{L,X}} 
{\Omega^L_*\left(X\right)}
\end{equation}
is an isomorphism.
\end{prop}
\begin{proof}
Since the ground field is of characteristic zero, the
unipotent radical $H^u$ of $H$ is split over $k$. Now the proof is exactly
same as the proof of Proposition~\ref{prop:red-torus}, where we just have to 
replace $B$ and $T$ by $H$ and $L$ respectively.
\end{proof}

\subsection{The motivic cobordism theory}
Before we prove our main results of this section, we recall the theory
of motivic algebraic cobordism $MGL_{*,*}$ introduced by Voevodsky in
\cite{Voevodsky}. This is a bi-graded ring cohomology theory in the category
of smooth schemes over $k$. Levine has recently shown in \cite{Levine1}
that $MGL_{*,*}$ extends uniquely to a bi-graded oriented Borel-Moore homology
theory $MGL'_{*,*}$ on the category of all schemes over $k$. This homology
theory has exterior products, homotopy invariance, localization exact sequence 
and Mayer-Vietoris among other properties ({\sl cf.} [{\sl loc. cit.}, 
Section~3]). Moreover, the universality of Levine-Morel cobordism theory 
implies that there is a unique map 
\[
\vartheta : \Omega_* \to MGL'_{2*, *}
\]
of oriented Borel-Moore homology theories. Our motivation for studying the
motivic cobordism theory in this text comes from the following result of
Levine.
\begin{thm}[\cite{Levine2}]\label{thm:LCOMP}
For any $X \in \sV_k$, the map $\vartheta_X$ is an isomorphism.
\end{thm}

We draw some simple consequences of working with the rational 
coefficients.
We recall from \cite{Levine1} that for a smooth $k$-variety $X$, there
is a Hopkins-Morel spectral sequence 
\begin{equation}\label{eqn:FG-action1}
E^{p,q}_2(n) = CH^{n-q}(X, 2n-p-q) \otimes \bL^q \Rightarrow MGL^{p+q,n}(X)
\end{equation}
which is the algebraic analogue of the Atiyah-Hirzebruch spectral sequence
in complex cobordism.

If $X$ is possibly singular, we embed it as a closed subscheme of a smooth
scheme $M$. Then, the functoriality of the above spectral sequence with
respect to an open immersion yields a spectral sequence
\[
E^{p,q}_2(n) = CH^{n-q}_X(M, 2n-p-q) \otimes \bL^q \Rightarrow MGL^{p+q,n}_X(M)
\]
of cohomology with support. Since the higher Chow groups and the motivic
cobordism groups of $M$ with support in $X$ are canonically isomorphic to the
higher Chow groups and the Borel-Moore motivic cobordism groups of $X$
({\sl cf.} \cite{Bloch}, \cite[Section~3]{Levine2}), the above spectral
sequence is identified with 
\begin{equation}\label{eqn:FG-action1S}
E_{p,q}^2(n) = CH_{n}(X, p) \otimes \bL^q \Rightarrow MGL'_{2n+p+q,n+q}(X).
\end{equation}

Now, suppose that a finite group $G$ acts on $X$. By embedding $X$
equivariantly in a smooth $G$-scheme $M$, the formula
\[
MGL'_{p,q}(X) : = \Hom_{{\sS}{\sH}(k)}\left({\Sigma}^{\infty}_TM/{(M-X)}, 
{\Sigma}^{p', q'} MGL\right)
\]
(where $p' = 2{\rm dim}(M)-p, q' = {\rm dim}(M)-q$)
shows that $G$ acts naturally
on $MGL'_{p,q}(X)$. It also acts on the higher Chow groups $CH_p(X,q)$ likewise.

We also observe that $G$ in fact acts on the cycle
complex $\sZ_j(X, \cdot)$ by acting trivially on $\Delta^{\bullet}$. This action
is given by 
\[
(g, \sigma) \mapsto \mu^*_g(\sigma)
\]
where $\mu_g$ is the automorphism of $X \times \Delta^i$ 
associated with $g \in G$ and $\sigma$ is an
irreducible admissible cycle on $X \times \Delta^i$. This action extends
linearly to all $\sigma \in \sZ_j(X, i)$. Since $G$ acts trivially on
$\Delta^{\bullet}$, it preserves the boundary map of the complex
$\sZ_j(X, \cdot)$. In particular, $G$ acts on the cycle complex of $X$. If 
we let ${\sZ_j(X, \cdot)}^G$ denote the subcomplex of the invariant cycles, 
then the exactness of the functor ``$G$-invariants'' on the category
of $\Q[G]$-modules implies that 
\begin{equation}\label{eqn:EXACTM}
H_i\left({\sZ_j(X, \cdot)}^G\right) \xrightarrow{\cong}
{CH_j(X, i)}^G.
\end{equation} 

Moreover, as the $E^1$-terms of the spectral sequence ~\eqref{eqn:FG-action1S}
are given by $E^1_{p,q} = \sZ_{n}(X, p) \otimes \bL^q$ with differential
$d^1 : E^1_{p,q} \to E^1_{p-1,q}$, we see that the $E^1$ is just the cycle complex
of $X$ and hence it is a complex with $G$-action as described above.
In other words, the $E^1$ terms are complexes of $\Q[G]$-modules.
We conclude that the spectral sequence ~\eqref{eqn:FG-action1S} is equipped
with a natural $G$-action. Since we are dealing with cohomology with
rational coefficients, the exactness of the functor of taking $G$-invariants
on the category of $\Q[G]$-modules implies that the above spectral
sequence descends to the spectral sequence of `$G$-invariants'
\begin{equation}\label{eqn:FG-action1S1}
{E'}_{p,q}^2(n) = \left(CH_{n}(X, p)\right)^G \otimes \bL^q \Rightarrow 
\left(MGL'_{2n+p+q,n+q}(X)\right)^G.
\end{equation}       
As a consequence, we get the following.
\begin{lem}\label{lem:G-INV*C}
Let $G$ be a finite group acting freely on a $k$-scheme $X$ with quotient
$Y$. There is an isomorphism
\[
MGL'_{p,q}(Y) \xrightarrow{\cong} \left(MGL'_{p,q}(X)\right)^G.
\]
\end{lem}
\begin{proof}
This follows immediately from the spectral sequences ~\eqref{eqn:FG-action1S}
and ~\eqref{eqn:FG-action1S1}, combined with \cite[Corollary~8.3]{Krishna}.
\end{proof}

Recall that a connected and reductive group $G$ over $k$ is said to be
{\sl split}, if it contains a split maximal torus $T$ over $k$ such that
$G$ is given by a root datum relative to $T$. One knows that every
connected and reductive group containing a split maximal torus is
split ({\sl cf.} \cite[Chapter~XXII, Proposition~2.1]{SGA3}).  
In such a case, the normalizer $N$ of $T$ in $G$ and all its connected
components are defined over $k$ and the quotient $N/T$ is the Weyl group
$W$ of the corresponding root datum.

\begin{lem}\label{lem:FG-action-S}
Let $G$ be a connected reductive group and let $T$ be a split maximal torus
in $G$. Put $H= G/N$, where $N$ is the normalizer of $T$ in $G$. Then, any
\'etale locally trivial $H$-fibration $f: X \to Y$ over a smooth variety $Y$
induces an isomorphism
$f^*: MGL^{*,*}(Y) \xrightarrow{\cong} MGL^{*,*}(X)$.
\end{lem}
\begin{proof}
This follows immediately from the similar result for the higher Chow groups
in \cite[Lemma~3.7]{Krishna} and the spectral sequence
~\eqref{eqn:FG-action1}.
\end{proof}

\begin{prop}\label{prop:FG-action}
Let the algebraic group $G$ be as in Lemma~\ref{lem:FG-action-S} and let 
$(V, U)$ be a good pair for the $G$-action corresponding to $j \ge 0$. Then 
for any $X \in \sV_G$, the pull-back map 
\begin{equation}\label{eqn:FGaction**}
\Omega_*\left(X \stackrel{G}{\times} U\right) \xrightarrow{f^*_X}
\Omega_*\left(X \stackrel{N}{\times} U\right)
\end{equation}
is an isomorphism.
\end{prop}
\begin{proof}
Since $G$ acts on $X$ linearly, we can find a $G$-equivariant closed 
embedding $X \inj M$, where $M \in \sV^S_G$. Let $W$ be the complement of
$X$ in $M$. Set $X_G = X \stackrel{G}{\times} U$. We have similar meaning
for $X_N$. Then $X_G$ is a closed subscheme of $M_G$ with complement
$W_G$. Moreover, $M_N \xrightarrow{f_M} M_G$ is an \'etale locally trivial 
$G/N$-fibration.
This yields a commutative diagram of long exact sequences
\begin{equation}\label{eqn:FG-ac*}
\xymatrix@C.6pc{
MGL^{*,*}(M_G) \ar[r] \ar[d]_{f^*_M} &
MGL^{*,*}(W_G) \ar[r] \ar[d]^{f^*_W} & MGL^{*,*}_{X_G}(M_G) \ar[r] \ar[d]^{f^*_X} &
MGL^{*,*}(M_G) \ar[r] \ar[d]^{f^*_M} & MGL^{*,*}(W_G) \ar[d]^{f^*_W} \\
MGL^{*,*}(M_N) \ar[r] & MGL^{*,*}(W_N) \ar[r] & MGL^{*,*}_{X_N}(M_N) \ar[r] &
MGL^{*,*}(M_N) \ar[r] & MGL^{*,*}(W_N).}
\end{equation}
Since $M_G$ and $W_G$ are smooth, the pull-back maps $f^*_M$ and $f^*_W$ are
isomorphisms by Lemma~\ref{lem:FG-action-S}. This implies in particular
that $f^*_X$ is also an isomorphism. Since 
\[
\Omega_*(X_G) \cong MGL'_{2*, *}(X_G) : = 
MGL^{2{\rm dim}(M_G)- 2*, {\rm dim}(M_G)-*}_{X_G}(M_G)
\]
by Theorem~\ref{thm:LCOMP} and since the pull-back $f^*_X : \Omega_*(X_G)
\to \Omega_*(X_N)$ is independent of the choice of the embedding,
we conclude that the map in ~\eqref{eqn:FGaction**} is an isomorphism.
\end{proof}

\begin{thm}\label{thm:W-inv}
Let $G$ be a connected linear algebraic group and let $L$ be a Levi subgroup
of $G$ with a split maximal torus $T$. Let $W$ denote the Weyl group of $L$
with respect to $T$.
Then for any $X \in \sV_G$, the natural map 
\begin{equation}\label{eqn:W-inv1}
\Omega^G_*(X) \to {\left(\Omega^T_*(X)\right)}^W
\end{equation}
is an isomorphism.
\end{thm}
\begin{proof}
By Proposition~\ref{prop:NRL}, we can assume that $G = L$ and hence $G$ is a 
connected reductive group with split maximal torus $T$. 

Let $N$ denote the normalizer of $T$ in $G$ so that $W = N/T$ and
$H = G/N$. Fix $i \in \Z$. We choose a sequence of
$l_j$-dimensional good pairs $\{(V_j,U_j)\}$ as in 
Theorem~\ref{thm:NO-Niveu} for the $G$-action.
Then, this is also a sequence of good pairs for the action of $N$ and $T$. 
Setting $X^j_G = X \stackrel{G}{\times} U_j$, we see that there
is an \'etale locally trivial $H$-fibration $X^j_N \to X^j_G$. Hence by 
Proposition~\ref{prop:FG-action},
the smooth pull-back $\Omega_{i+{l_j}-g}(X^j_G) \to \Omega_{i+{l_j}-n}(X^j_N)$ is an
isomorphism, where $n$ is the dimension of $N$. Taking the inverse limit and
using Theorem~\ref{thm:NO-Niveu}, we get an isomorphism
\begin{equation}\label{eqn:W-inv2*}
\Omega^G_i(X) \xrightarrow{\cong} \Omega^N_i(X).
\end{equation}

On the other hand, as $W$ acts freely on $X_T$ with quotient $X_N$, 
it follows from  Theorem~\ref{thm:LCOMP} and Lemma~\ref{lem:G-INV*C} 
that the natural map 
\begin{equation}\label{eqn:W-inv3}
\Omega_{i+{l_j}-n}(X_N) \to \left(\Omega_{i+{l_j}-n}(X_T)\right)^W
\end{equation}
is an isomorphism.
Since the action of $W$ on the inverse system 
${\left\{\Omega_{i+{l_j}-n}(X_T)\right\}}_j$ induces the similar action on the 
inverse limit and since the inverse limit commutes with taking the 
$W$-invariants, we get
\begin{equation}\label{eqn:W-inv3}
{\underset{j}\varprojlim} \ \Omega_{i+{l_j}-n}(X_N) \xrightarrow{\cong}
\left({\underset{j}\varprojlim} \ \Omega_{i+{l_j}-n}(X_N)\right)^W.
\end{equation}
Since the left and the right terms are same as $\Omega^N_i(X)$ and
$\Omega^T_i(X)$ respectively by choice of our good pairs and 
Theorem~\ref{thm:NO-Niveu}, we conclude that $\Omega^N_i(X)
\xrightarrow{\cong} \left(\Omega^T_i(X)\right)^W$. We complete the proof
of the theorem by combining this with ~\eqref{eqn:W-inv2*}.
\end{proof}

\begin{cor}\label{cor:W-split}
Let  $X \in \sV_G$ be as in Theorem~\ref{thm:W-inv}. Then the natural
map 
\begin{equation}\label{eqn:Borel3*}
{\Omega_*^G\left(X\right)} \xrightarrow{r^G_{T,X}} 
{\Omega_*^T\left(X\right)} 
\end{equation}
is a split monomorphism which is natural for the morphisms in ${\sV}_G$. 
In particular, if $H$ is
any closed subgroup of $G$, then there is a split injective map 
\begin{equation}\label{eqn:Borel3}
{\Omega_*^H\left(X\right)} \xrightarrow{r^G_{T,X}} 
{\Omega_*^T\left(G \stackrel{H}{\times} X\right)}. 
\end{equation}
\end{cor}
\begin{proof}
The first statement follows directly from Theorem~\ref{thm:W-inv}, where
the splitting is given by the trace map. The second statement follows from 
the first and Proposition~\ref{prop:Morita}.
\end{proof} 

Before we draw some more consequences, we have the following topological 
analogue of Theorem~\ref{thm:W-inv}, which is much simpler to prove.
Recall from ~\eqref{eqn:ECompC} that if $G$ is a complex Lie group and
$X$ is a finite $CW$-complex with a $G$-action, then its {\sl equivariant
complex cobordism} is defined as
\begin{equation}\label{eqn:ECompC-R} 
MU^*_G(X) := MU^*\left(X \stackrel{G}{\times} EG\right).
\end{equation}

\begin{thm}\label{thm:W-inv-Top}
Let $G$ be a complex Lie group with a maximal torus $T$ and Weyl group $W$.
Then for any $X$ as above, the natural map 
\begin{equation}\label{eqn:W-inv1-T}
MU^*_G(X) \to {\left(MU^*_T(X)\right)}^W
\end{equation}
is an isomorphism.
\end{thm}
\begin{proof} As in the proof of Theorem~\ref{thm:W-inv}, we can reduce to the
case when $G$ is reductive.
It follows from the above definition of the equivariant complex cobordism
and the similar definition of the equivariant singular cohomology of $X$,
plus the Atiyah-Hirzebruch spectral sequence in topology that there is
a spectral sequence 
\begin{equation}\label{eqn:W-inv2-T}
E^{p,q}_2 = H^p_G(X, \Q) \otimes_{\Q} MU^q \Rightarrow MU^{p+q}_G(X).
\end{equation}
Since the Atiyah-Hirzebruch spectral sequence degenerates rationally, we see
that the above spectral sequence degenerates too. Since one knows that
$H^*_G(X) \cong \left(H^*_T(X)\right)^W$ ({\sl cf.}
\cite[Proposition~1]{Brion1}), the corresponding result for the cobordism
follows.
\end{proof}

\begin{thm}\label{thm:ACC*}
For a connected linear algebraic group $G$ over $\C$,
the map $\rho^G: \Omega^*(BG) \to MU^*(BG)$ ({\sl cf.} 
Corollary~\ref{cor:Alg-Comp*}) is an isomorphism 
of $\bL$-algebras. In particular, the natural map
\[
CH^*(BG) \xrightarrow{{\ov{\rho}}^G} MU^*(BG) \widehat{\otimes}_{\bL} \Q
\]
is an isomorphism of $\Q$-algebras.
\end{thm}
\begin{proof} To prove the first isomorphism, we can use 
Theorems~\ref{thm:W-inv} and ~\ref{thm:W-inv-Top} to reduce to the case of a
torus. But this is already known even with the integer
coefficients ({\sl cf.} \eqref{eqn:BT1} and \cite{Totaro1}).
The second isomorphism follows from the first and Proposition~\ref{prop:CBCH}.
\end{proof}

\begin{remk}\label{remk:ACC*R}
The map ${\ov{\rho}}^G: CH^*(BG) \to MU^*(BG) \widehat{\otimes}_{\bL} \Z$ was 
discovered by Totaro in \cite{Totaro1} even before Levine and Morel discovered 
their algebraic cobordism. It is 
conjectured that the map ${\ov{\rho}}^G$ is an isomorphism with the integer
coefficients for a connected complex algebraic group $G$. Totaro modified this 
conjecture to an expectation that  ${\ov{\rho}}^G$ should be an isomorphism 
after localization at a prime $p$ such that $MU^*(BG)_{(p)}$ is concentrated in
even degree. The above theorem proves the isomorphism in general with
the rational coefficients. We also remark that the map
$MU^*(BG) \widehat{\otimes}_{\bL} \Q \to H^*(BG)$ is an isomorphism
({\sl cf.} \cite{Totaro2}). The above result then shows that the cycle class
map for the classifying space is an isomorphism with the rational
coefficients. One wonders if the techniques of this paper could be applied
to the the algebraic version of the Brown-Peterson cobordism theory to
prove the Totaro's modified conjecture. 
\end{remk}


\begin{thebibliography}{99}
\bibitem{Bloch} S. Bloch, {\sl Algebraic cycles and higher $K$-theory\/},
Advances in Math., {\bf 61}, (1986), 267-304. \
\bibitem{BO} S. Bloch, A. Ogus, {\sl Gersten's conjecture and the homology of 
schemes\/},  Ann. Sci. \'Ecole Norm. Sup., (4),  {\bf 7}, (1975),
181-201 \
\bibitem{Borel} A. Borel, {\sl Linear Algebraic groups\/}, Second
edition, GTM {\bf 8}, Springer-Verlag, (1991). \
\bibitem{Brion1} M. Brion, {\sl Equivariant cohomology and equivariant
intersection theory\/}, Representation theory and algebraic geometry,
NATO ASI series, {\bf C514}, Kluwer Academic Publishers, (1997). \
\bibitem{Brion2} M. Brion, {\sl Equivariant Chow groups for torus actions\/},
Transform. Groups,  {\bf 2},  (1997),  no. 3, 225-267. \
\bibitem{SGA3} M. Demazure, A. Grothendieck, {\sl Sch{\'e}mas en Groupes\/},
Lecture Notes in Math., {\bf 153}, (1970), Springer-Verlag. \
\bibitem{DD} D. Deshpande, {\sl Algebraic Cobordism of Classifying Spaces\/},
MathArxiv, mathAG/0907.4437, (2009). \
\bibitem{EG} D. Edidin, W. Graham, {\sl  Equivariant intersection theory\/}, 
Invent. Math., {\bf 131}, (1998), 595-634. \
\bibitem{EDR} D. Edidin, W. Graham, {\sl Riemann-Roch for equivariant Chow 
groups\/},  Duke Math. J.,  {\bf 102},  (2000),  no. 3, 567-594. \  
\bibitem{GIT} J. Fogarty, F. Kirwan, D. Mumford, {\sl Geometric Invariant
Theory\/}, 3rd Edition, Springer-Verlag, (1994). \
\bibitem{Fulton} W. Fulton, {\sl Intersection Theory\/}, 2nd edition,
Springer-Verlag, (1998). \ 
\bibitem{Joshua} R. Joshua, {\sl Algebraic $K$-theory and higher Chow 
groups of linear varieties\/},  Math. Proc. Cambridge Philos. Soc.,
{\bf 130},  (2001),  no. 1, 37-60. \
\bibitem{KK} V. Kiritchenko, A. Krishna, {\sl Cobordism ring of certain
spherical varieties\/}, In preparation. \
\bibitem{Kock} B. Kock, {\sl Chow motif and higher Chow theory of $G/P$\/},
Manuscripta Math., {\bf 70},  (1991),  no. 4, 363-372. \
\bibitem{Krishna} A. Krishna, {\sl Riemann-Roch for equivariant
$K$-theory\/}, MathArxiv, mathAG/0906.1696, (2009). \
\bibitem{Krishna1} A. Krishna, {\sl Equivariant $K$-theory and higher
Chow groups of smooth varieties\/}, MathArxiv, math.AG/0906.3109, (2009). \
\bibitem{Krishna4} A. Krishna, {\sl Equivariant cobordism of schemes\/},
MathArxiv, mathAG/1007.3176, (2010) \ 
\bibitem{Krishna5} A. Krishna, {\sl Equivariant cobordism of schemes with torus
action\/}, In preparation. \
\bibitem{Krishna3} A. Krishna, {\sl Cobordism of flag bundles\/},
MathArxiv, mathAG/1007.1083, (2010). \
\bibitem{KU} A. Krishna, V. Uma, {\sl Cobordism rings of toric varieties\/},
In preparation. \
\bibitem{Landweber} P. Landweber, {\sl Elements of infinite filtration in 
complex cobordism\/},  Math. Scand.,  {\bf 30},  (1972), 223-226. \ 
\bibitem{Levine1} M. Levine, {\sl Oriented cohomology, Borel-Moore homology, 
and algebraic cobordism\/}, Special volume in honor of Melvin Hochster.  
Michigan Math. J.,  {\bf 57},  (2008). \ 
\bibitem{Levine2} M. Levine,  {\sl Comparison of cobordism theories\/},
J. Algebra,  {\bf 322},  (2009),  no. 9, 3291-3317. \  
\bibitem{LM1} M. Levine, F. Morel, {\sl Algebraic cobordism I\/},
Preprint, www.math.uiuc.edu/$K$-theory/0547, (2002). \
\bibitem{LM} M. Levine, F. Morel, {\sl Algebraic cobordism\/},
Springer Monographs in Mathematics. Springer, Berlin, (2007). \
\bibitem{LP} M. Levine, R. Pandharipande, {\sl Algebraic cobordism revisited\/},
Invent. Math.,  {\bf 176},  (2009),  no. 1, 63-130. \
\bibitem{PPR} I. Panin, K. Pimenov, O. R\"ondings, {\sl On the relation of 
Voevodsky's algebraic cobordism to Quillen's $K$-theory\/},  Invent. Math.,
{\bf 175},  (2009),  no. 2, 435-451. \
\bibitem{Quillen} D. Quillen, {\sl Elementary proofs of some results of 
cobordism theory using Steenrod operations\/},  Advances in Math.,  {\bf 7},
(1971), 29-56. \ 
\bibitem{Springer} T. Springer, {\sl Linear algebraic groups\/}, 
Second edition, Progress in Math., {\bf 9}, (1998), Birkhauser. \
\bibitem{Sumihiro} H. Sumihiro, {\sl Equivariant completion II\/},
J. Math. Kyoto, {\bf 15}, (1975), 573-605. \
\bibitem{Thomason1} R. Thomason, {\sl Equivariant algebraic vs. topological
$K$-homology Atiyah-Segal-style\/}, Duke Math. J., {\bf 56}, 
(1988), 589-636. \
\bibitem{Totaro2} B. Totaro, {\sl Torsion algebraic cycles and complex 
cobordism\/},  J. Amer. Math. Soc.,  {\bf 10},  (1997),  no. 2, 467-493. \ 
\bibitem{Totaro1} B. Totaro, {\sl The Chow ring of a classifying space\/},
Algebraic $K$-theory (Seattle, WA, 1997), 
Proc. Sympos. Pure Math., Amer. Math. Soc., {\bf 67}, (1999), 249-281. \ 
\bibitem{VV} G. Vezzosi, A. Vistoli, {\sl Higher algebraic $K$-theory for 
actions of diagonalizable groups\/},  Invent. Math.,  {\bf 153},  (2003),  
no. 1, 1-44. \
\bibitem{Voevodsky} V. Voevodsky, {\sl $\A^1$-homotopy theory\/},
Proceedings of the International Congress of Mathematicians, {\bf 1},
(Berlin, 1998), Doc. Math., (1998), 579-604. \
\end{thebibliography}
\end{document}